\documentclass[reqno,10pt]{amsart}
\usepackage{amsmath}
\usepackage{amssymb}
\usepackage{amsfonts}
\usepackage{graphicx}
\usepackage[abs]{overpic}
\usepackage{caption}
\usepackage{subcaption}
\usepackage{wrapfig}
\usepackage{float}
\usepackage{graphicx}
\usepackage{physics}
\usepackage{esint} 
\usepackage{braket} 

\newcommand{\HH}{\mathcal{H}} 

\usepackage{tikz}
\usetikzlibrary{decorations.pathreplacing,calc}
\def\centerarc[#1](#2)(#3:#4:#5)
{ \draw[#1] ($(#2)+({#5*cos(#3)},{#5*sin(#3)})$) arc (#3:#4:#5); }

\definecolor{blue_links}{RGB}{13,0,180} 

\usepackage{hyperref}
\hypersetup{
    colorlinks=true, 
    linktoc=all,     
    linkcolor=blue_links,  
    citecolor=blue_links,
    urlcolor=blue_links,
}

\usepackage{mathtools}

\usepackage{xcolor}

\newtheorem{theorem}{Theorem}[section]
\newtheorem{lemma}[theorem]{Lemma}
\newtheorem{proposition}[theorem]{Proposition}

\newtheorem{remark}[theorem]{Remark}

\newtheorem*{theorem*}{Theorem}

\newcommand{\R}{\mathbb{R}}

\newcommand{\BBB}{\color{black}} 
 
\newcommand{\EEE}{\color{black}}

\def\eps{\varepsilon}

\def\dist{\operatorname{dist}}
\def\Xint#1{\mathchoice
    {\XXint\displaystyle\textstyle{#1}}%
    {\XXint\textstyle\scriptstyle{#1}}%
    {\XXint\scriptstyle\scriptscriptstyle{#1}}%
    {\XXint\scriptscriptstyle\scriptscriptstyle{#1}}%
\!\int}
\def\XXint#1#2#3{{\setbox0=\hbox{$#1{#2#3}{\int}$}
\vcenter{\hbox{$#2#3$}}\kern-.5\wd0}}

\def\dashint{\Xint-}
\newcommand{\sm}{\setminus}

\usepackage{color}

\newcommand{\ks}[1]{{\textcolor{black}{#1}}} 
\newcommand{\cl}[1]{{\textcolor{black}{#1}}} 

\numberwithin{equation}{section}

\begin{document} 

\title[Dimension of the singular set for Griffith almost-minimizers]
{Hausdorff dimension of the singular set \\ for Griffith almost-minimizers in the plane}
\author[M. Friedrich]{Manuel Friedrich} 
\address[Manuel Friedrich]{Department of Mathematics, Johannes Kepler Universit\"at Linz. Altenbergerstrasse 66, 4040 Linz,
    Austria}
\email{manuel.friedrich@jku.at}
\author[C. Labourie]{Camille Labourie}
\address[Camille Labourie]{Université de Lorraine, CNRS, IECL, F-54000 Nancy, France}
\email{camille.labourie@univ-lorraine.fr}
\author[K. Stinson] {Kerrek Stinson} 
\address[Kerrek Stinson]{Department of Mathematics, University of Utah, Salt Lake City, UT, USA}
\email{kerrek.stinson@utah.edu}


\begin{abstract}

    We consider regularity of the crack set associated to a minimizer of the Griffith fracture energy, often used in modeling brittle materials. We show  that the crack is uniformly rectifiable  which in conjunction with our previous epsilon-regularity result \cite{FLSeps} allows us to prove that the singular set has dimension strictly less than $1$. This size estimate also applies to almost-minimizers. As a byproduct, we prove   higher integrability for the gradient of local minimizers of the Griffith energy, providing a positive answer to the \BBB analog \EEE of De Giorgi's conjecture \cite{DeGiorgi} {for} the Mumford--Shah functional.
\end{abstract}

\maketitle

\section{Introduction}
\BBB For \ks{over a} century, the pioneering work by {\sc Griffith} \cite{griffith} \ks{has formed} the pillar for  understanding the behavior of brittle materials. His fundamental idea, namely that the propagation of crack \ks{is due to the} competition between stored elastic energy and dissipation related to an infinitesimal increase of the crack, \ks{was recast} in a variational setting by   {\sc Francfort and Marigo}   \cite{Francfort-Marigo:1998} (see also with {\sc Bourdin} in \cite{Bourdin-Francfort-Marigo:2008}).  This approach leads to the minimization of the so-called  \emph{Griffith energy} which, in the \ks{planar} setting, takes the form 
\begin{equation}\label{eq:energy}
    \int_{\Omega\setminus K} \mathbb{C}e(u)\colon e(u) \, \dd{x} + \mathcal{H}^1(K),
\end{equation}
where $\Omega \subset \R^2$ denotes the \emph{reference configuration}, $K \subset \Omega$ the closed \emph{crack set},   $u \colon \Omega \setminus K \to \R^2$ denotes the \emph{displacement} with    symmetrized gradient  $e(u) : = (\nabla u + \nabla u^T)/2$, \EEE and $\mathbb{C}$ represents a fourth-order   tensor of {elasticity} constants.   In analogy to the Mumford--Shah conjecture \cite{David}, \BBB one expects \EEE that within $\Omega$ the crack of a minimizer should consist of smooth curves, meeting at triple junctions or terminating in crack tips. This full conjecture is currently beyond reach since, a priori, there may be other singular points with far more erratic structure. Nevertheless, in this paper  we prove that the number of singular points has dimension strictly less than $1$.

Interest in free discontinuity problems like \eqref{eq:energy} was first driven by applications in image segmentation. In special cases, information from these simpler models carries over to the Griffith energy. Precisely, under the assumption of antiplanar shear, the energy (\ref{eq:energy}) reduces to the classical \emph{Mumford--Shah functional} arising in image segmentation \cite{Ambrosio-Fusco-Pallara:2000}, where the results of {\sc Ambrosio} \cite{Ambrosio:90}, {\sc De Giorgi et al}.\  \cite{deGiorgiCarrieroLeaci}, and {\sc Ambrosio et al}.\  \cite{Ambrosio-Fusco-Pallara:2000} show that a weak minimizer exists, the crack of a minimizer is closed, and \BBB the crack \EEE  coincides with a smooth {submanifold} at most points. An alternative, more variational, path was subsequently developed by {\sc David} and coauthors, see \cite{David} and references therein, to analyze the functional.

In contrast, to account for linearized elasticity in the vectorial problem \eqref{eq:energy},  a host of new techniques \BBB was required \EEE -- even just to obtain existence of minimizers. 
Building on the  function space  $BD$ (bounded deformation), see e.g.\  \cite{ACD}, {\sc Dal Maso} \cite{DalMaso:13} introduced the space $GSBD$, which is essentially the set of displacements with finite Griffith energy. 
As the Griffith energy has no direct coercivity on $u$, initial results for existence of minimizers were conditional, requiring $L^\infty$-bounds on the functions or fidelity terms.
In a series of papers \cite{Friedrich:15-5,Friedrich:15-4,Solombrino}, the first author introduced a new strategy whereby minimizers were found by deriving compactness for modified displacements. This result relied on the \emph{piecewise  Korn inequality}, a powerful functional inequality that is only available in dimension two.
 Using an alternative approach based only on a \emph{Korn inequality for functions with small jump set}  \cite{CCS, Chambolle-Conti-Francfort:2014, Conti-Iurlano:15.2, Friedrich:15-3}, {\sc Chambolle and Crismale} \cite{Crismale} provided an existence result for \BBB minimizers attaining Dirichlet boundary data \EEE in arbitrary dimension. 
Turning towards regularity of the crack set, {\sc Conti et al}.\  \cite{Conti-Focardi-Iurlano:19} showed that, \BBB in dimension two, \EEE the crack is (locally) closed, \BBB which \EEE is a basal regularity result ensuring that the crack does not pulverize the material. Mathematically, this relies on showing \emph{Ahlfors-regularity} of the crack, a uniform concentration estimate of the form
\begin{equation*}
    \frac{1}{C_{\rm Ahlf}} r \leq \mathcal{H}^1(K \cap B(x,r)) \leq C_{\rm Ahlf} r 
\end{equation*}  
for some constant $C_{\rm Ahlf} \ge 1  $ and arbitrary open balls $B(x,r)\subset \Omega$ with center $x \in K$ and radius $r>0$. 
{\sc Chambolle et al}.\  \cite{Iu3} extended this result to higher dimensions, and later {\sc Chambolle and Crismale}~\cite{CrismaleCalcVar} accounted for Dirichlet boundary conditions. Our work in \cite{FLS2} shows that Ahlfors-regularity even holds for almost-minimizers of a wide class of anisotropic and heterogeneous energies.

The first results on $C^{1,\alpha}$-regularity of the crack were obtained by {\sc Babadjian et al}.\ \cite{BIL} in the plane and by the second author and \cl{{\sc Lemenant}} \cite{CL} in arbitrary dimension. While serious advances, both of these epsilon-regularity results relied on topological assumptions \BBB regarding connectedness of the crack \EEE  that are not a priori satisfied by the crack of a minimizer. 
At its core, \BBB such assumptions \EEE circumvent the lack of a coarea formula for $BD$-functions.
A subsequent epsilon-regularity result by \cl{the present authors} \cite{FLSeps} showed that, at least in the plane, these topological constraints could be removed. Implicit in our work was that the singular set of the crack of a Griffith minimizer is $\mathcal{H}^1$-null.

In this paper,   we show two primary results that advance the understanding of  minimizers $(u,K)$ of the   Griffith energy (\ref{eq:energy}) in the plane,  and, in particular, give fine information on the structure of the crack:
\begin{itemize}
    \item We show that minimizers are \emph{uniformly rectifiable}, which means that the crack is locally contained in a curve with controlled geometry.

    \item Tying together our epsilon-regularity result \cite{FLSeps} and the global information carried by uniform rectifiability, we prove a \emph{porosity statement} that can then be used to show that there is a singular set $K^*\subset K$ with dimension \textit{strictly less} than $1$, such that $K\setminus K^*$ is a $C^{1,1/2 }$  curve. \EEE
\end{itemize}

Below we discuss these results at a formal level. For simplicity, we restrict our discussion to minimizers within the introduction. However, we point out that most of our results also apply to almost-minimizers which extends our results to a variety of models related to \eqref{eq:energy}, see \cite[Subsection 2.1]{FLSeps} for examples.

\noindent \textbf{Uniform rectifiability.}
Our goal is to  show that minimizers of the Griffith energy in the plane have a crack $K$ that is \emph{locally uniformly rectifiable}, i.e.,   locally $K$ is  a closed Ahlfors-regular set which is contained in the image $\Gamma(\R)$ of a parametrization $\Gamma \colon \R \to \R^2$ with good geometric properties.
Heuristically, one can think of $\Gamma$ as a Bilipschitz mapping that additionally allows for more general behavior such as cusps or self-intersections.
This is stronger than rectifiability as it relies on a single curve  (instead of countably many) with properties of the parametrization giving quantitative control on $K$ at all scales and locations (instead of {in the blow-up limit} at almost every point).

The first definition of uniform rectifiability appeared in the 90's as a way to characterize the $d$-dimensional sets $K$ of $\R^N$ on which the Calderon-Zygmund kernels define a bounded operator on $L^2(K,\HH^d)$ \cite{DavidKZ}.
It turned out that  uniform rectifiability is an useful notion of regularity for many problems involving sets. Equivalent,  more convenient  {characterizations}  were introduced that do  not rely on covering the set with a single parameterization.
In particular, a one-dimensional set is uniformly rectifiable if and only if in each ball $B(x,r)$ with $x \in K$,  the set $K$ contains a substantial subset of a compact connected set (or a `big piece of a connected set' in the usual terminology), see  Theorem \ref{thm: david} below for details. Notice that this is indeed a quantitative rectifiability property since purely unrectifiable sets are precisely those whose intersection with any compact connected set of finite length is negligible. The fundamental insight behind the big piece of connected sets for Griffifth minimizers is that, in most balls, the crack is `almost connected' in the sense that the size of the gaps is uniformly small.

Our approach to recover this kind of  `connectedness' is modeled on a related argument of {\sc David and Semmes} \cite{DS,DS2} for the Mumford--Shah functional, see  also \cite{David}. The essential difference between the Mumford--Shah functional and the Griffith energy is that, in the prior, the energy penalizes the full gradient, thereby making the \emph{coarea formula} a viable tool. At this point, there is no known \BBB analog \EEE of the coarea formula for the symmetrized gradient as in (\ref{eq:energy}). In the Sobolev setting, the standard tool for gaining control of the full gradient with only the symmetrized gradient is Korn's inequality. Unfortunately, with a crack, the domain of application effectively becomes $\Omega\setminus K$ where one cannot naively apply the Korn inequality as the constants in the estimate depend on the regularity of the domain and thereby the a priori irregular set $K$. To overcome this difficulty, we use the piecewise  Korn inequality for $GSBD$  functions, introduced by the first author in \cite{Friedrich:15-4}, to prove uniform rectifiability. We formally state our result here and refer to Theorem \ref{thm:UniformRect} for the precise details.

\begin{theorem*}
    Let $\Omega\subset \R^2$ be a bounded Lipschitz domain, and let {$(u,K)$} be a minimizer of the Griffith energy (\ref{eq:energy}). Then, the crack $K$ is (locally) uniformly rectifiable.
\end{theorem*}

We highlight the essential elements of the proof of uniform rectifiability. Since the $2$-elastic energy can always be dominated by surface energy, with equiintegrability in mind, we can show that the (normalized) $p$-elastic energy is small along a significant portion of the crack for fixed $p\in (1,2)$. 
At one of these points, the crack must break the material into nearly rigid pieces. If the crack does not \textit{almost separate} these rigid pieces, the elastic energy must pay the price of continuously transitioning between significantly different values on either side of the crack, and this would contradict the small elastic energy. Here, almost separation  is quantified using the piecewise Korn inequality. 
We refer to Section \ref{sec:UniformRect} for the technical details.

\noindent \textbf{Size of the singular set and porosity.}
As hinted at before, one {expects} the singular set $K^*$ to consist of isolated points, such as triple junctions and crack tips.   Though this specific characterization is currently beyond reach, our techniques  provide a first step towards a fine characterization of the crack set, namely an upper bound for the dimension of $K^*$. In particular,  uniform rectifiability is an important ingredient as  it implies many good geometric properties. One of them is that $K$ is flat near many $x \in K$, \BBB where \EEE \emph{flatness} is a key condition within our epsilon-regularity theorem \cite{FLSeps} to deduce that $K$ is smooth.
Therefore, the uniform rectifiability is instrumental   for proving that in every ball $B(x,r)$ with  $x \in K$  there is a large piece of $K \cap B(x,r)$ that is a smooth  curve. \EEE
This latter property is referred to as the \emph{porosity} of the  singular points  $K^*$.

The porosity is then the principal ingredient to show that $K^*$ has dimension strictly less than $1$ and that $e(u) \in L^p_{\mathrm{loc}}$ for some $p > 2$. We   formally state this result as follows and refer to Theorem~\ref{th: main result} for precise details.

\begin{theorem*}
    Let $\Omega\subset \R^2$ be a bounded Lipschitz domain. Then, there are $\delta  >0$ and $p>2$ such that for any  minimizer  $(u,K)$   of the Griffith energy \eqref{eq:energy} there exists a {closed} set $K^* \subset K$ satisfying
    $${\rm dim}(K^*)<1- \delta  \quad \text{ and }\quad  \text{$K\setminus K^*$ is a  $C^{1,1/2}$-manifold\EEE} $$
along with 
    $$e(u) \in L^p_{\rm loc}(\Omega;  \R^{2\times 2}). \EEE  $$ 
\end{theorem*}

\cl{Following the Mumford--Shah conjecture and \cite[Conjecture 4]{DLF},}
we expect that $K^*$ {is composed of isolated points} and that $e(u) \in L^{4,\infty}_{\mathrm{loc}}$ but even for the Mumford--Shah functional in the planar setting, these are still major open questions.  For minimizers of the Mumford--Shah functional, the porosity property was {pioneered} by {\sc David} \cite{DavidC1} in dimension $2$, and generalized to every dimension by {\sc Maddalena and Solimini} \cite{MaSo} and {\sc Rigot} \cite{Rigot}.
{\sc Ambrosio, Fusco, and Hutchinson} \cite{AFH} observed that the dimension of $K^*$ can also be estimated by the integrability exponent of $\nabla u$. More precisely, if $(u,K)$ is a local minimizer in an open set $\Omega \subset \R^N$ and $\nabla u \in L^p_{\mathrm{loc}}(\Omega; \BBB  \R^{N}) \EEE$ for some $p > 2$, then
\begin{equation*}
    \mathrm{dim}(K^*) \leq \max(N-2,N-p/2).
\end{equation*}
{\sc De Giorgi}  conjectured  that local minimizers satisfy $\nabla u \in L^p_{\mathrm{loc}}$ for some $p > 2$,  see  \cite[Conjecture~1]{DeGiorgi}. {\sc De Lellis and Focardi} \cite{DLF} actually observed that the Mumford--Shah conjecture in dimension $2$ is equivalent to $\nabla u \in L^{4,\infty}_{\mathrm{loc}}$.
Therein they {also} proved the  higher integrability of the gradient (for a certain $p > 2$) in dimension $2$, and {\sc De Philippis and Figalli} \cite{DPF} generalized this to every dimension.

Though these connections have not yet been made for the Griffith energy, similar statements might be expected. It  indicates that our theorem on the size of the singular set $K^*$ is an important step towards understanding if a Mumford--Shah type conjecture holds in the present  setting.

\noindent \textbf{Organization of the paper.}
We briefly outline the organization of the paper. In Section \ref{sec:prelim}, we introduce the definitions regarding minimizers of the Griffith energy and state our main results. In Section \ref{sec: prelim}, we recall mathematical tools that will be used throughout the paper, such as the epsilon-regularity theorem,  sets of finite perimeter, and Korn inequalities. The proof of uniform rectifiability of the crack is within Section \ref{sec:UniformRect}. Finally, we estimate the size of the singular set and conclude the primary result of the paper in Section~\ref{sec: main result}.

\section{Basic notions and main results}\label{sec:prelim}
\subsection{Basic  notions}\label{sec: basic not}
We begin by introducing a variety of definitions that will be required to state our main results and are used repeatedly within the paper.

Throughout, we let $\Omega \subset \R^2$ be an open set.  We denote the open ball centered at $x\in \Omega$ with radius $r>0$ by $B(x,r)$. {The} two-dimensional Lebesgue and the one-dimensional Hausdorff measure are {denoted} by $\mathcal{L}^2$ and $\mathcal{H}^1$, respectively.  
The sets of symmetric and skew symmetric matrices are  denoted by $\R^{2 \times 2}_{\rm sym}$ and $ \R^{2 \times 2}_{\rm skew}\EEE$.  {The unit circle in $\R^2$ is given by $\mathbb{S}^1$.} We write $U \subset \subset \Omega$ if $\overline{U} \subset \Omega$. By $\chi_U$ we denote the characterisic function of $U$. The {Euclidean} distance of a point $y \in \R^2$ from a set $U$ is denoted by ${\rm dist}(y,U)$. We often use the notation $\pm$ as a placeholder for $+$ and $-$. {The average} integral over a set $U$ {is} denoted by $\dashint_U$. An affine map is called  an \textit{(infinitesimal) rigid motion} if its gradient is a skew-symmetric matrix, and is usually written as $a(x) = A\,x + b$ for $A \in \R^{2 \times 2}_{\rm skew}$ and $b \in \R^2$.  Given an open set ${U} \subset \R^2$, we also define $LD({U})$ as
\begin{equation*}
    LD({U}) := \Big\{ \text{$u \in W^{1,2}_{\mathrm{loc}}(U;  \R^2 \EEE )$ such that $\int_{{U}} \abs{e(u)}^2 \dd{x} < +\infty$} \Big\}.
\end{equation*}
Here, $e(u) := (\nabla u + \nabla u^T)/2$ denotes the \emph{symmetrized gradient} of $u$. In a similar fashion, we define $LD_{\rm loc}({U})$.

We  proceed with some basic notions in order to study almost-minimizers of the Griffith energy.

\noindent{\bf Admissible  pairs.}
We define an \emph{admissible pair} as a pair $(u,K)$ such that $K$ is a relatively closed subset of $\Omega$ and $u \in W^{1,2}_{\mathrm{loc}}(\Omega \setminus K;\R^2)$.

\noindent{\bf {Elasticity tensor and} Griffith energy.}
For {the entire paper,} we fix a linear  map $\mathbb{C} : \R^{2 \times 2} \to \R^{2 \times 2}_{\rm sym}$ such that
\begin{equation*}
    \mathbb{C}(\xi - \xi^T) = 0 \quad \text{and} \quad \mathbb{C} \xi : \xi \geq  c_0  \abs{\xi + \xi^T}^2 \quad \text{ for all } \xi \in \R^{2 \times 2}
\end{equation*}
for some $c_0 > 0$.
Given an admissible pair $(u,K)$ in $\Omega$ and $\ks{B} \subset \Omega$, we define the \emph{Griffith energy} of $(u,K)$ in $B$ as
\begin{align}\label{eq: main energy}
    \int_{B \setminus K} \mathbb{C} e(u) \colon e(u) \, {\rm d}x + \mathcal{H}^1(K \cap B). 
\end{align}

\noindent{\bf Competitors.}
Let $(u,K)$ be an admissible pair and let $x \in \Omega$ and $r > 0$ {be} such that ${B(x,r)} \subset \subset \Omega$. We say that $(v,L)$ is a \emph{competitor} of $(u,K)$ in $B(x,r)$ if  $(v,L)$ is  an admissible pair  such that
\begin{equation*}
    L \setminus B(x,r) = K \setminus B(x,r) \quad \text{and} \quad v = u \  \ \text{a.e. in} \ \Omega \setminus \big(K \cup B(x,r)\big).
\end{equation*}

\noindent{\bf Local minimizers and almost-minimizers.}
We call a \emph{gauge} a non-decreasing function $h\colon (0,+\infty) \to [0,+\infty]$ such that $\lim_{t \to 0^+} h(t) = 0$.  \BBB Our uniform rectifiability result holds for any gauge whereas for the regularity of almost-minimizers we use the explicit  form $h(t) = h(1) t^{\alpha}$, where $\alpha \in (0,1)$,  due to the application of the  epsilon-regularity result \cite {FLSeps}. \EEE  Moreover, we say that a pair $(u,K)$ is \emph{coral} provided that, for all $x \in K$ and for all $r > 0$,
\begin{equation}\nonumber
    \HH^{1}(K \cap B(x,r)) > 0.
\end{equation}
We say that an admissible, coral pair $(u,K)$ \BBB with locally finite energy \EEE is a \emph{Griffith local almost-minimizer with gauge $h$} in $\Omega$
if, given any  $x \in \Omega$ and  $r > 0$ with ${B(x,r)} \subset \subset \Omega$, all competitors $(v,L)$ of $(u,K)$ in $B(x,r)$ satisfy 
\begin{multline}
    \int_{B(x,r) \setminus K} \mathbb{C} e(u)\colon e(u)    \dd{x} + \HH^{1}\big(K \cap B(x,r)\big) \le  \int_{B(x,r) \setminus L} \mathbb{C} e(v)\colon e(v)    \dd{x} + \HH^{1}\big(L \cap B(x,r)\big) + h(r) r. \label{amin}
\end{multline}
In the case $h = 0$, we say that $(u,K)$ is a \emph{Griffith local minimizer}. In the following, we frequently omit the words `local' and `gauge', and  refer only to Griffith almost-minimizers for simplicity. 

We remark that the coral assumption is used simply  to choose a good representative for the crack $K$, analogous to choosing a good representative of a Sobolev function. In particular, it ensures that  (\ref{eqn:AhlforsReg}) below holds for \textit{all} points $x\in K$. \ks{We also refer to \cite[Subsection 2.1]{FLSeps} for further examples of relevant almost-minimizers.}

\noindent{\bf Ahlfors regularity.}
By the results in \cite{FLS2}, if $(u,K)$ is an almost-minimizer of the Griffith energy with any gauge $h$, then $K$ is (locally) Ahlfors-regular. We remark that  this is proven for Griffith minimizers in \cite{Iu3,  CrismaleCalcVar, Conti-Focardi-Iurlano:19}.  Precisely, {\cite{FLS2} shows} that there exist  constants $C_{\rm Ahlf} \geq 1$ and $\varepsilon_{\rm Ahlf}  \in  (0,1) \EEE $ (depending only on $\mathbb{C}$) such that for all $x \in K$, $r > 0$ with $B(x,r) \subset \Omega$ and $h(r) \leq \varepsilon_{\rm Ahlf}$, we have
\begin{equation}\label{eqn:AhlforsReg}
    \frac{1}{C_{\rm Ahlf}}r \leq \mathcal{H}^1\big(K\cap B(x,r)\big)\leq {C_{\rm Ahlf}}r.
\end{equation}
As a direct consequence {of \BBB the \EEE definition}, we mention that almost-minimizers satisfy
\begin{equation}\label{eqn:AhlforsReg2}
    \int_{B(x,r) \setminus K}  |e(u)|^2 \, \dd x     \leq {\bar{C}_{\rm Ahlf}}r  
\end{equation}
for all $x\in K$ and $r>0$  with  $B(x,r) \subset \Omega$ and $h(r) \le \varepsilon_{\rm Ahlf}$,  where $\bar{C}_{\rm Ahlf}$ depends only on $\mathbb{C}$. Indeed, this follows from \eqref{amin} by comparing $(u,K)$ with $v = u \chi_{\Omega \setminus B(x,\rho)}$  and $L = (K \setminus B(x,\rho)) \cup \partial B(x,\rho)$ {for $\rho < r$, and then letting $\rho \to r$}. (If $B(x,r) \subset \subset \Omega$, one can directly take $\rho = r$.)

\noindent \textbf{Uniform rectifiability.}  {Here we adapt}  \cite[Definition 27.17]{David}. We say that  $K$ is \emph{locally uniformly  {rectifiable}} {in $\Omega$ with gauge $h$} and constants $C_{\rm ur} \geq 1$ and $\eps_{\rm ur}>0$ provided that for all $x \in K$ and  for all $r > 0$ with $B(x,2r) \subset \Omega$ and $h(2r) \le \eps_{\rm ur}$  it holds
\begin{equation}\label{eq_defi_ur}
    \text{$K \cap B(x,r)$ is contained in an Ahlfors-regular curve with constant $C_{\rm ur}$,}
\end{equation}
where we say that a set is an \emph{Ahlfors-regular curve} with constant $C > 0$ if it is a set of the form ${\Gamma}(\R)$, where ${\Gamma} \colon \R \to \R^2$ is a Lipschitz function such that
\begin{equation}\label{eq_z1}
    \abs{{\Gamma}(y) - {\Gamma}(z)} \leq \abs{y - {z}} \quad \text{for all  $y,{z} \in \R$}
\end{equation}
and
\begin{equation}\label{eq_z2}
    \HH^1\big(\set{{z} \in \R \colon \,  {\Gamma}({z}) \in B(y,r)}\big) \leq C r \quad \text{for $y \in \R^2$ and $r > 0$}.
\end{equation}
The important point of the definition is that, in contrast to the usual definition of {rectifiability}, $K \cap B(x,r)$ is contained in a \emph{single curve}  and the properties of the function $\Gamma$  give a quantitative control on $K \cap B(x,r)$ at \emph{all scales and locations}.  
{Here, the notion is \emph{local} because it only holds in balls $B(x,r)$  well contained in $\Omega$  with a controlled gauge, \EEE {and} the parametrization in (\ref{eq_defi_ur}) might depend on $B(x,r)$. A (globally)  uniformly rectifiable set would be  entirely contained in a single Ahlfors-regular curve. Still, in our definition  \eqref{eq_defi_ur}--\eqref{eq_z2}  it is important that the constant in \eqref{eq_z2}  does not depend on $B(x,r)$, as this shows that we have a {uniform} control on the geometric structure at all scales {and locations}.}

\subsection{Main results}

We let $\Omega \subset \R^2$ be an open set. We recall the definitions of the Griffith energy in \eqref{eq: main energy} and  of almost-minimizers in \eqref{amin}. We now formulate the main result of the paper.

\begin{theorem}[Regularity of almost-minimizers]\label{th: main result}
    (i) For each choice of exponent $\alpha \in (0,1)$, there exists a constant $\delta  > 0$ (depending on $\mathbb{C}$ and $\alpha$) such that the following holds.
    If $(u,K)$ is an almost-minimizer of the Griffith energy \eqref{eq: main energy} with gauge $h(t) = h(1) t^\alpha$,  there exists a {closed} set $K^* \subset K$ of Hausdorff dimension  at most {$1- \delta  $} such that $K$ is a $C^{1,\alpha/2}$-{graph} in {a} neighborhood of each $x \in K \setminus K^*$.

    (ii) Moreover, if $(u,K)$ is a local minimizer ($h \equiv 0$), it holds that $e(u) \in L^{p}_{\rm loc}(\Omega;\R^{2 \times 2})$ for some constant $p >2$ that depends on $\mathbb{C}$. More precisely, there exist  constants $C \geq 1$ and $\varepsilon > 0$ (both depending $\mathbb{C}$) such that for all $x \in \Omega$ and  $r > 0$ with $B(x,2r) \subset \Omega$,
    \begin{equation}\label{eq: almost last}
        \int_{B\left(x,r\right)} \! \abs{e(u)}^{{p}} \dd{x} \leq C r^{2-{p}/2}.
    \end{equation}
\end{theorem}

We note that within the introduction we stated the theorem for $\alpha  =1 $ for minimizers, which is possible using \cite[Remark 2.5(i)]{FLSeps}. Our main result is (in part) a consequence of uniform rectifiability, {which is of interest in its own right.}

\begin{theorem}[Uniform rectifiability]\label{thm:UniformRect}
    Let $(u,K)$ be an almost-minimizer of the Griffith energy \eqref{eq: main energy} with respect to a gauge $h$. Then, $K$ is locally uniformly rectifiable {in $\Omega$ with gauge $h$} and constants $C_{\rm ur} \geq 1$  and \EEE $\varepsilon_{\rm ur} > 0$ (depending  only on $\mathbb{C})$.
\end{theorem}

Combining  the epsilon-regularity result (see  Theorem \ref{th: eps_reg} below) \EEE  and Theorem \ref{thm:UniformRect} guarantees that the subset $K^*$ of nonregular points of $K$ is \emph{porous}, i.e., that $K$ is $C^{1,\alpha/2}$-regular in \cl{most} balls.
This property then allows us to conclude the statement in Theorem \ref{th: main result} by an abstract argument given in  \cite[Section~51]{David} or \cite[Lemma~5.8]{DaSe}. The important point is that uniform rectifiability allows us to show that, in many balls, \BBB the flatness $\beta_K$ (see \eqref{eq_beta} below) \ks{is below a fixed threshold (taken to be $\eps\ll 1$) thereby initializing one piece of epsilon-regularity's hypotheses}.

\section{Preliminaries}\label{sec: prelim}

\BBB In this section, we \EEE introduce mathematical tools that will be used in the proof of our main results.

\subsection{Epsilon-regularity theorem}
We state our epsilon-regularity result which requires that in a neighborhood $B(x_0,r_0)$ of a jump point $x_0 \in K$, $r_0 \ll 1$,  the set $K \cap B(x_0,r_0)$  is suitably flat and has a lower bound on the jump height at $x_0$, related to the difference of the values of $u$ in half balls around $x_0$. Implicitly, it will also rely on the elastic energy inside $B(x_0,r_0)$ being suitably small. To this end, we introduce three basic quantities that will  be frequently used, namely the normalized elastic energy, the flatness, and the normalized jump.

\noindent{\bf The  normalized $p$-elastic energy.} Let $(u,K)$ be an admissible pair and $p\in (1,2]$. For any $x_0 \in \Omega$ and $r_0 > 0$ such that $B(x_0,r_0) \subset  \Omega$, we define the \emph{normalized $p$-elastic energy} of $u$ in $B(x_0,r_0)$ as
\begin{equation}\label{eqn:pelastic}
            \omega_{p}(x_0,r_0) : = \left(r_0^{p/2}  \BBB \dashint_{B(x_0,r_0)\setminus K} \EEE |e(u)|^p\, {\rm d} x\right)^{2/p}.
        \end{equation}

\noindent{\bf The flatness.} 
{{Let $x_0 \in \R^2$, $r_0 > 0$, and $K$} be a relatively closed subset of $B(x_0,r_0)$ containing $x_0$.}
We define the {\it  flatness} of $K$ in $B(x_0,r_0)$ by
\begin{equation}\label{eq_beta}
    \beta_K(x_0,r_0) := \frac{1}{r_0} \inf_\ell \sup_{y \in K \cap B(x_0,r_0)}{\rm dist }(y,\ell),
\end{equation}
where the infimum is taken over all lines $\ell$ through $x_0$.
An equivalent definition is that $\beta_K(x_0,r_0)$ is the infimum of all $\varepsilon > 0$ for which there exists a line $\ell$ through $x_0$ such that
\begin{equation*}
    K \cap B(x_0,r_0) \subset \set{y \in B(x_0,r_0) \colon \,  \mathrm{dist}(y,\ell) \leq \varepsilon r_0}.
\end{equation*}
It is easy to check that the infimum in (\ref{eq_beta})  is attained. When there is no ambiguity, we often write $\beta(x_0,r_0)$ instead of $\beta_K(x_0,r_0)$. For each $(x_0,r_0)$, we choose   (possibly not uniquely) a  unit normal vector  $\nu(x_0,r_0) \in \mathbb{S}^1$ of an approximating line $\ell$ in the sense of \eqref{eq_beta}. For $t >0$, we define the  sets \EEE
\begin{align}\label{eq :Bx0}
    D_{t}^\pm(x_0,r_0)  := \big\{ x \in B(x_0,r_0) \colon     \pm  (x - x_0) \cdot \nu(x_0,r_0) > t \big\},  
\end{align}
where we note that $K \cap (D_{t}^+(x_0,r_0) \cup D_{t}^-(x_0,r_0)) = \emptyset$ for all $t \ge \beta_K(x_0,r_0)r_0$.  We use the letter $D$ to remind the reader of the shape of the sets. We frequently omit writing $(x_0,r_0)$ if no confusion arises.

Note that there are many different notions of flatness, see e.g.\ \cite[Equation (8.2)]{Ambrosio-Fusco-Pallara:2000}. More precisely, $\beta_K$ could be called \emph{unilateral} flatness.

\noindent{\bf The normalized jump.}
{Let $(u,K)$ be an admissible pair. For $x_0 \in K$, $r_0 > 0$ such that $B(x_0,r_0) \subset \Omega$ and $\beta_K(x_0,r_0) \le 1/2$}, \EEE we define the \emph{mean infinitesimal rotations and translations} on   $D_{\beta_K(x_0,r_0)r_0}^\pm$ (omitting $(x_0,r_0)$) by 
\begin{align}\label{eq: optimal values}
    A^\pm(x_0,r_0) = \dashint_{D_{\beta_K(x_0,r_0)r_0}^\pm} \frac{1}{2} (\nabla u - \nabla u^T) \dd{x}, \quad b^\pm(x_0,r_0) = \dashint_{D_{\beta_K(x_0,r_0)r_0}^\pm}  \big(u(x) - A^\pm(x_0,r_0) x\big)  \dd{x}.
\end{align}
We define the \emph{normalized jump} of $u$ in $B(x_0,r_0)$ by 
\begin{equation}\label{def:normalizedJump}
    J_u(x_0,r_0) := \frac{1}{\sqrt{r_0}}  \inf_{y \in K \cap B(x_0,r_0) }  \big|(A^+(x_0,r_0) - A^{-}(x_0,r_0)) y + (b^+(x_0,r_0) - b^{-}(x_0,r_0)  )\big|, 
\end{equation}
which corresponds to the minimal jump between the  rigid motions $y \mapsto  A^\pm(x_0,r_0)y + b^\pm(x_0,r_0)$ {along the portion of the crack $K \cap { B(x_0,r_0)}$}. \BBB We often \EEE  drop the subscript $u$ if no confusion arises.  
\BBB Note that \EEE we inherit a more involved definition of the normalized jump from \cite{FLSeps} as compared to regularity results for the Mumford--Shah functional, see for instance \cite{Lemenant}. \BBB This is \EEE due to the invariance of the energy under  rigid motions. We point out that a notion of this type is not necessary for epsilon-regularity results of Griffith  under a connectivity or separation property, as used in \cite{BIL} or \cite{CL}. For us,  \BBB besides its relevance in \cite{FLSeps}, \cl{the normalized jump} will \EEE be instrumental to derive porosity properties, see Subsection \ref{sec: porosity} for details.

\begin{remark}[Normalization]\label{rem: normalization}
{\normalfont

    The above quantities are called \emph{normalized} as they are invariant under rescaling. Indeed, given an admissible pair $(u,K)$ and a ball $B(x_0,r_0) \subset \Omega$, the pair {$(\tilde u,\tilde K)$ in $B(0,1)$} defined by $\tilde u(x):=  {r_0^{-1/2} u(x_0 + r_0 x)}  $ and $\tilde K := {r_0^{-1}(K-x_0)}$ satisfies
$$\omega_{p}(x_0,r_0) = \omega_{\tilde u,p}(0,1), \quad \beta_K(x_0,r_0) = \beta_{\tilde K}(0,1), \quad J_u(x_0,r_0) = J_{\tilde u}(0,1) ,$$ where $\omega_{\tilde u,p}$ is the $p$-elastic energy of $\tilde u$. We emphasize that this is the natural rescaling of the problem as $(\tilde u, \tilde K)$ is an almost-minimizer with gauge $\tilde{h}(\cdot) := h(r_0 \cdot)$ in $B(0,1)$.

}
\end{remark}

\begin{remark}[Scaling and shifting properties for $\beta$]\label{rmk_beta}
    \normalfont
    For all  $x_0 \in K$  and  $0 < r \leq r_0$ with  $B(x_0,r_0) \subset \Omega$, it holds 
    \begin{equation*}
        \beta_K(x_0,r) \leq \frac{r_0}{r} \beta_K(x_0,r_0),
    \end{equation*}
    and for all {balls} $B(x,r) \subset B(x_0,r_0)$ with $x \in K$,   one can easily check that  \EEE
    \begin{equation*}
        \beta_K(x,r) \leq \frac{2r_0}{r} \beta_K(x_0,r_0).
    \end{equation*}
    This follows directly from the definition. 
\end{remark}

With these quantities in hand, we \BBB recall \EEE the epsilon-regularity result from \cite[Theorem 2.4]{FLSeps}.

\begin{theorem}[Epsilon-regularity]\label{th: eps_reg} 
    For each choice of exponent $\alpha \in (0,1)$, there exist  $\eps_0 > 0$ and $\gamma \in (0,1)$ (both depending on $\mathbb{C}$ and $\alpha$) such that the following holds.
    Let $(u,K)$ be an almost-minimizer of the Griffith energy \eqref{eq: main energy} with gauge $h(t) = h(1) t^\alpha$. 
    For all $x_0 \in K$ and  $r_0 > 0$ such that $B(x_0,r_0) \subset \Omega$ and
    \begin{align}\label{eq: smallli epsi}
        \beta_K(x_0,r_0) + J_u(x_0,r_0)^{-1} + h(r_0) \le \eps_0, 
    \end{align}
    the set $K \cap B(x_0,\gamma r_0)$ is a   $C^{1,\alpha/2}${-graph}.
\end{theorem}

\BBB  For the proof of Theorem \ref{th: main result},  we will apply this result in many balls. For this, both quantities $\beta_K$ and $J_u$ need to be initialized. For $\beta_K$ this is achieved by exploiting  uniform rectifiability, see \eqref{eq: 5.28} below,   for $J_u$ the details can be found in Lemma \ref{lem_init_J}. \EEE

\subsection{Estimates on the difference of rigid motions via the $p$-elastic energy}

While the proof of Theorem \ref{th: eps_reg} only relied on the $2$-elastic energy \BBB (see particularly \cite[Proposition~4.7]{FLSeps}), \EEE uniform rectifiability will require dropping to $p\in (1,2)$ to gain further integrability properties. In this subsection, we recall some  estimates  already derived for general $p \in (1,2]$ in \cite{FLSeps}. Note that for all balls $B(x,r) \subset B(x_0,r_0)$ we have
        \begin{equation}\label{eqn:naiveUpEst}
            \omega_p(x,r) \leq  {\left(\frac{r_0}{r}\right)^{4/p - 1}} \omega_p(x_0,r_0).
        \end{equation}
        {Now suppose $\beta(x_0,r_0) \leq 1/2$, and recall} the definition of the sets $D^\pm_{\beta(x_0,r_0)r_0}$ in \eqref{eq :Bx0}  (omitting $(x_0,r_0)$ in the notation) {along with} the fact that    $K \cap (D_{\beta(x_0,r_0)r_0}^+\cup D_{\beta(x_0,r_0)r_0}^-) = \emptyset$.  With $A^\pm(x_0,r_0)$ and $ b^\pm(x_0,r_0)$ as in \eqref{eq: optimal values}, we denote their  corresponding  rigid motions by
        \begin{align*}
            a^\pm_{x_0,r_0}(y) = A^\pm(x_0,r_0)\,y + b^\pm(x_0,r_0) \quad \text{for $y \in \R^2$}. 
        \end{align*} 
    When there is no confusion, we will drop the explicit dependence on $x_0$ and $r_0$ in the notation.
    We state an estimate for the difference of rigid motions, which implies control on the difference of $J$ \BBB (see \eqref{def:normalizedJump}) \EEE on balls of different size, see \cite[Lemma 4.17]{FLSeps}. 

    \begin{lemma}[Balls of different size]\label{lem_aA}
        Let $p\in (4/3,2].$ Let $x_0 \in K$ and  $r_0 > 0$ be such that $B(x_0,r_0) \subset \Omega$ and   {$h(r_0) \leq \varepsilon_{\rm Ahlf}$.} Let $\gamma \in (0,1)$ be such that $\beta(x_0,t) \leq 1/8$ for all $t \in [\gamma r_0,r_0]$.
        Then, there exists a constant $C \geq 1$, depending on $p$ and $\mathbb{C}$, such that for all $r \in [\gamma r_0,r_0]$ we have
        \begin{equation}\label{eqn:bAControl}
            | a^\pm_{x_0,r_0}(y) -  a^\pm_{x_0,r}(y)|    \leq C r_0^{1/2}\omega_p(x_0,r_0)^{1/4} \quad \quad   \text{for all $y \in B(x_0,r)$}.   
        \end{equation}
    \end{lemma}

    \begin{remark}[Varying centers]\label{rem: vary}
        \normalfont
        As noted in \cite[Remark 4.18]{FLSeps}, for given  $B(x_0,r_0) \subset \Omega$  and another ball $B(x,r)\subset B(x_0,r_0)$ with $x_0,x \in K$ and $\beta(x_0,r_0) \leq r_0/(16 r)$, we also have 
        $$  { | a^\pm_{x_0,r_0}(y) -  a^\pm_{x,r}(y)|    \leq C_*  r_0^{1/2}\omega_p(x_0,r_0)^{1/4} \quad \quad   \text{for all $y \in B(x,r)$},} $$
and       
        $$   | A^\pm_{x_0,r_0} -  A^\pm_{x,r}|    \leq C_*  r_0^{-1/2}\omega_p(x_0,r_0)^{1/4} $$
        where $C_*$  depends on $r_0/r$.
    \end{remark}
    \EEE

\EEE

\subsection{Piecewise Korn-Poincar\'e inequality}\label{korn-sec}

In this subsection, we collect some results for sets of finite perimeter and affine maps, and we recall the piecewise Korn inequality proved in \cite{Solombrino}. 

\textbf{Sets of finite perimeter.}  For a set of finite perimeter ${P}$, we denote by $\partial^* {P}$ its \emph{essential boundary} and by ${P}^{(1)}$ its \emph{points of density $1$}, see \cite[Definition 3.60]{Ambrosio-Fusco-Pallara:2000}.  By ${\rm diam}({P}) = {\rm ess \, sup}\lbrace |x-y| \colon x,y \in P \rbrace$ we denote the essential diameter of $P$. \EEE  A set of finite perimeter $P$ is called \emph{indecomposable} if it cannot be written as   {$P_1 \cup P_2$ with $P_1 \cap P_2 = \emptyset$, $\mathcal{L}^{ 2 }(P_1), \mathcal{L}^{ 2 }(P_2) >0$, and $\mathcal{H}^{ 1 }(\partial^* P) = \mathcal{H}^{ 1 }(\partial^* P_1) + \mathcal{H}^{ 1 }(\partial^* P_2)$}.  This notion generalizes the concept of  connectedness to sets of finite perimeter.  
We will need the following two lemmas on sets of finite perimeter.

\begin{lemma}\label{lemma: diam}
    Let ${P} \subset \R^2$ be a bounded set that has finite perimeter and is indecomposable. Then,  ${\rm diam}({P}) \le \mathcal{H}^1(\partial^* {P})$.
\end{lemma}

\begin{lemma}[Relative isoperimetric inequality]\label{lemma: maggi}
    Let ${U} \subset \R^2$ {be an} open, bounded {set} with Lipschitz boundary. Then, for all  sets $P \subset  U  $ with finite perimeter one has
    \begin{align*}
        {\rm (i)}  \ \ \mathcal{L}^2({P}) > \frac{1}{2} \mathcal{L}^2({U}) \ \ \ \text{ or } \ \ \ 
        {\rm (ii)}  \ \ {\mathcal{H}^1(\partial^* P)} \le C   \mathcal{H}^1(\partial^* {P} \cap {U})
    \end{align*}
    for some constant $C > 0$ only depending on $U$.
\end{lemma}

The proof of Lemma \ref{lemma: diam} can be found in \cite[Proposition 12.19, Remark 12.28]{maggi}. Noting that $\mathcal{H}^1(\partial^* P\cap \partial \ks{U})\leq C{\rm diam}(P)$ for some constant $C>0$ depending only on $U$, Lemma \ref{lemma: maggi} is a consequence of \cite[Lemma 4.6]{Solombrino}.  (Strictly speaking, \EEE if $P$ is not indecomposable,  we apply \cite[Lemma 4.6]{Solombrino}  for its indecomposable components.)   \EEE

We also recall the  structure theorem \ks{for} the boundary of planar sets   $E$  of finite perimeter in  \cite[Corollary~1]{Ambrosio-Morel}: there exists a unique countable decomposition  of $\partial^* E$ into pairwise almost disjoint rectifiable Jordan curves. Here, we  say that $\Gamma \subset \R^2$ is  a rectifiable Jordan curve if $\Gamma = {\Gamma}([a,b])$ for some $a < b$, {where by an abuse of notation the parametrization of $\Gamma$ is a  Lipschitz continuous map also denoted by $\Gamma$ that is one-to-one on $[a,b)$ and such that  $\Gamma(a) = \Gamma(b)­$.}

\textbf{Caccioppoli partitions.} We say that a partition $(P_j)_{j \ge 1}$ of an open set $U\subset \R^2$ is a \textit{Caccioppoli partition} of $U$ if  $\sum\nolimits_j \mathcal{H}^1(\partial^* P_j\cap U) < + \infty$.   The  local structure of Caccioppoli partitions can be characterized as follows, see \cite[Theorem 4.17]{Ambrosio-Fusco-Pallara:2000}.
\begin{theorem}\label{th: local structure}
    Let $(P_j)_j$ be a Caccioppoli partition of $U$. Then 
    $$\bigcup\nolimits_j P_j^{(1)} \cup \bigcup\nolimits_{i \neq j} (\partial^* P_i \cap \partial^* P_j)$$
    contains $\mathcal{H}^{1}$-almost all of $U$, where $P^{(1)}$ denote the points of density $1$.
\end{theorem}

  \textbf{Korn inequalities.} \EEE   As mentioned, we will rely on a  Korn-Poincar\'e inequality for $GSBD$ functions.  There has  been a variety of efforts to overcome a loss of control due to the combination of frame indifference and fracture, see e.g.\ \cite{CCS, Chambolle-Conti-Francfort:2014, Conti-Iurlano:15.2, Friedrich:15-3,  Friedrich:15-4}.  
  Of these options, the \emph{piecewise Korn inequality}, proven   in dimension two, see  \cite[Theorem 2.1]{Friedrich:15-4} for the case $p=2$ and \cite[Remark 5.6]{Friedrich:15-4} for general $p \in (1,\infty)$, provides the best control. 
To prove the uniform rectifiability, we use a refined variant of the piecewise Korn inequality, introduced for  the study of quasistatic crack evolution \cite{Solombrino}. Compared to \cite[Theorem~2.1]{Friedrich:15-4}, it refines the estimate on the partition in the sense  that the nearly rigid `pieces' have boundaries nearly coinciding with the crack.  We state the result for $GSBD^p$-functions \cite{DalMaso:13}, but the reader can simply replace $u \in GSBD^p(\Omega)$ by an admissible pair $(u,K)$, with the relation $J_u = K$.  We {recall} that $a\colon \R^2 \to \R^2$ is an (infinitesimal) \emph{rigid motion} if $a$ is affine and its \emph{symmetrized gradient} satisfies $ e(a) = ( \nabla a + (\nabla a)^{T})/2  = 0$.\EEE

\begin{proposition}[\BBB Piecewise Korn inequality\EEE]\label{thm:refinedKorn}
    Let $\Omega\subset \R^2$ be an open, bounded set with Lipschitz boundary, $0<\theta< 1$, and $p\in (1,\infty)$. Then, there exists  a \EEE constant  $C_{\theta}$  depending on $p$, $\Omega,$ and $\theta$   such that for $u \in GSBD^p(\Omega;\R^2)$ there is $u^\theta \in SBV(\Omega;\R^2)$ satisfying   
    \begin{align*}
        \mathcal{L}^2(\{u\neq u^\theta\}) &\leq \theta \left(\mathcal{H}^1(J_u) + \mathcal{H}^1(\partial \Omega) \right)^{2} 
        \end{align*}
        and there is  a finite Caccioppoli partition $(P_j)_{j=0}^{J}$ with  $P_0 = \{u \neq u^\theta\}$ and corresponding   rigid motions $a_j = A_j (\cdot) +   b_j $ such that 
        \begin{align*}
            \sum\nolimits_{j=0}^J \mathcal{H}^1\big(\partial^* P_j \cap (\Omega \sm J_u)\big)  &\leq \theta \left(\mathcal{H}^1(J_u) + \mathcal{H}^1(\partial \Omega) \right),  \\
            \|u^\theta - \sum\nolimits_{j=0}^J a_j \chi_{P_j}\|_{L^\infty(\Omega)} + \|\nabla u^\theta - \sum\nolimits_{j=0}^J A_j \chi_{P_j}\|_{L^1(\Omega)}&\leq C_\theta  \|e(u)\|_{L^p(\Omega)}.
        \end{align*}
        It holds that  $C_\theta\to \infty$ as $\theta \to 0$.
    \end{proposition}

    We note that this result was not explicitly stated for $p \neq 2$ but it also holds in this setting. In fact, once \cite[Theorem 2.1]{Friedrich:15-4} for general $p$ is applied, the proof of the  refinement described in \cite[Theorem 4.1]{Solombrino} works for any positive number $\mathcal{E}$ on the right-hand side of \cite[Equations (21), (22)]{Solombrino}.

\section{Uniform rectifiability: Proof of Theorem \texorpdfstring{\ref{thm:UniformRect}}{2.7}} \label{sec:UniformRect}

This section is devoted to the proof of Theorem \ref{thm:UniformRect}.
It is difficult to prove directly that $K$ is contained in  an  Ahlfors-regular curve.  As a remedy, we  take advantage of a characterization of uniform rectifiability by big pieces of connected sets. \EEE

\begin{theorem}\label{thm: david}
    Let $(u,K)$ be a Griffith almost-minimizer \BBB  with respect to a gauge $h$. \EEE Then,  there exist  constants $C_{\rm ur} \geq 1$ and $\eps_{\rm ur} > 0$ both depending on $\mathbb{C}$ such that for all $x_0 \in K$ and $r_0 > 0$ with {$B(x_0,r_0) \subset \Omega$ and $h(r_0) \leq \eps_{\rm ur}$}, we can find {a compact connected} set $\Gamma$ such that
    $$\mathcal{H}^1(\Gamma) \le C_{\rm ur} r_0 \quad \quad  \text{and} \quad \quad \mathcal{H}^1\big(K \cap \Gamma \cap B(x_0,r_0)   \big) \ge \frac{1}{C_{\rm ur}} r_0.  $$
\end{theorem}

{With the above theorem in hand, we may directly apply {\cite[Section 32]{David} (see also \cite[Theorem 31.5]{David})} to obtain Theorem \ref{thm:UniformRect}. We note that the parameters $\eps_{\rm ur}$ and $C_{\rm ur}$ are possibly changed to satisfy the definition of uniformly rectifiable within \eqref{eq_defi_ur}. We now focus our   attention on the proof of Theorem \ref{thm: david}.}

{As a preparation for the proof of Theorem \ref{thm: david},}  we show that we can initialize the $p$-elastic energy $\omega_p$ introduced in \eqref{eqn:pelastic}, in the sense  that, \BBB for any \EEE  $x_0 \in K$ and $B(x_0,r_0) \subset \Omega$, we find suitable balls $B(y,r) \subset B(x_0,r_0)$ such that $\omega_p(y,r) \le \eps$.  This is achieved by combining the following \ks{two}  abstract results. 

\begin{lemma}\label{lemma: david}
 Suppose that $K$ is a relatively closed Ahlfors-regular subset of $B(x_0,r_0) \subset \R^2$ \BBB and let $C_{\rm Ahlf}$ denote its Ahlfors-regularity constant.
    Suppose that there exists a function $u\in LD( B(x_0,r_0) \setminus K)$ such that $\int_{B(x,r)\setminus K} |e(u)|^2  \, \dd z \le C_0r$ for all $B(x,r) \subset B(x_0,r_0)$ for \cl{some} $C_0>0$. Let $p \in [1,2)$. Then, there exists $C_p >0$ depending on $C_0$, ${p}$, and  \BBB $C_{\rm Ahlf}$  \EEE such that  
    $$\int_{K \cap B(x_0,r_0/3)} \Big( \int_{0}^{r_0/3} \omega_p(y,t)   \frac{ \, \dd t }{t} \Big) \, \dd \mathcal{H}^1(y) \le C_p r_0. $$
\end{lemma}

For a proof we refer to \cite[Theorem 23.8]{David}, where we observe by   \cite[Remark 23.16]{David} that the proof actually works for any function $g\colon B(x_0,r_0) \to [0,\infty)$ with  $\int_{B(x,r)\setminus K} g(z)^2  \, \dd z \le C_0r$ for all $B(x,r) \subset B(x_0,r_0)$.

\begin{lemma}\label{lemma: cheb}
    {Suppose that $K$ is a relatively closed Ahlfors-regular subset of $B(x_0,\ks{r}) \subset \R^2$ and let $C_{\rm Ahlf}$ denote its Ahlfors-regularity constant.}
    Suppose that there exists a function $\psi \colon  ( B(x_0,r)\cap K) \times [0,r]  $ such that   
    $$\int_{K \cap B(x_0,r)} \Big( \int_0^r \psi(y,t) \frac{ \dd t }{t} \Big) \, \dd \mathcal{H}^1(y)  \le \bar{C} r $$
    for some $\bar{C}>0$. Then for all $\eps>0$ there exists {$c_\eps \in (0,1)$}, depending on $\eps$, $\bar{C}$ and $C_{\rm Ahlf}$ such that there are $y \in K \cap B(x_0,r)$ and $t \in [c_{\eps}r,r] $   with 
    $$ \psi(y,t) \le \eps. $$ 
\end{lemma}

\begin{proof}
    This is a standard Chebyshev argument, used for instance  in \cite[Corollary 23.38]{David}. We report the argument here for convenience of the reader.  
    {Suppose by contradiction that for fixed $\eps>0$, for all $c_\eps>0$ the statement fails. In particular, for any choice of} $c_\eps>0,$   
    we {have} $\psi(y,t)>\eps$ for all $y \in K \cap B(x_0,r)$ and $t \in [c_{\eps}r,r]$, and thus
    \begin{align*}
        \bar{C}r &  \ge \int_{K \cap B(x_0,r)}  \Big( \int_0^{r} \psi(y,t) \frac{ \dd t }{t}  \Big) \, \dd \mathcal{H}^1(y) \ge  \int_{K \cap B(x_0,r)} \Big( \int_{c_{\eps}r }^{r} \eps \frac{ \dd t }{t} \Big) \, \dd \mathcal{H}^1(y)  \\
                 & = \eps \mathcal{H}^1\big(K \cap B(x_0,r)\big) \left(\int_{c_{\eps}r }^{r} \frac{1}{t} \dd{t}\right) \ge \eps  C_{\rm Ahlf}^{-1}   r   \, {\rm log}( (c_{\eps})^{-1}),
    \end{align*}
    where we used \eqref{eqn:AhlforsReg}. 
    For $c_\eps$ small enough (depending {on $\varepsilon$}, $C_{\rm Ahlf}$, and $\bar{C}$) the above estimate yields a contradiction. 
\end{proof}

In the proof, we will also use the following general property of Ahlfors-regular sets.

\begin{lemma}[Lemma 23.25 of \cite{David}]\label{lemm: another david}
    {Suppose that $K$ is a relatively closed and Ahlfors-regular subset of $B(x_0,2r_0) \subset \R^2$, with constant $C_{\rm Ahlf}$.}
    Then, there is a constant $\bar{C} \geq 1$ only depending on $C_{\rm Ahlf}$ such that for any   $\gamma \in (0,1)$ we have
    $$\mathcal{L}^2\big(\big\{z \in B(  x_0, \EEE r)\colon \, {\dist}(z,K)\leq \gamma r\big\}\big) \leq \bar{C}\gamma r^2 \quad \text{ for all $0 < r \le r_0$}.$$
\end{lemma}

With this, we can now prove Theorem \ref{thm:UniformRect} in two parts.
The first part  (Steps 1--5) shows that, in a ball of low elastic energy, the rigid motions used to approximate $u$ in disjoint regions away from the crack have a distance bounded from below.
In the second part  (Steps 6--7),  we use the  \BBB piecewise Korn \EEE  inequality (Proposition \ref{thm:refinedKorn}) to show that we may separate the regions of the domain with different values of $u$ by the crack plus a vanishingly small part connecting the crack.

\begin{proof}[Proof of Theorem \ref{thm: david} (and Theorem \ref{thm:UniformRect})]
    We let $x_0 \in K$ and $r_0 >0$ be such that $ B(x_0,r_0) \EEE \subset \Omega$.  
    We start with fixing some parameters we use throughout the proof. {First, choose $\eps_{\rm ur}>0$ (as in the theorem statement) sufficiently small such that 
        \begin{align}\label{eq:h000}
            {h(r_0)} \leq \eps_{\rm ur}\le \min\Big\{\frac{1}{32}C_{\rm Ahlf}^{-1}, \eps_{\rm Ahlf} \Big\},
        \end{align}
    where  $C_{\rm Ahlf}$  and $\eps_{\rm Ahlf}$  are the constants \EEE in (\ref{eqn:AhlforsReg}), so that in particular \eqref{eqn:AhlforsReg} holds on $B(x_0,r_0)$.}
    We fix the constant  $c_A :=  1/(\BBB 128 \EEE C_{\rm Ahlf})$ and define $0 < \delta \le \min\lbrace c_A/(8\bar{C}), 1\rbrace$, where $\bar{C}$ is the constant from Lemma~\ref{lemm: another david}. \BBB Then, \EEE we  further select $0 < \mu \le \delta/4$ such that 
    \begin{align}\label{clear2}
        2c_A + 6 \pi \mu/\delta \leq \frac{1}{32}C_{\rm Ahlf}^{-1}.
    \end{align}
    {The constants $\delta$ and $\mu$  depend  only on $C_{\rm Ahlf}$ and thus on $\mathbb{C}$. (The choice becomes clear below in \eqref{clear}.)}
    Fix any $p\in [1,2)$. Fix $\eps \in (0,1)$ to be chosen sufficiently small later, depending only on $\mu$, $\delta$, $p$, and  $\mathbb{C}$  {(see below in \eqref{clear}, \eqref{eq: covers a large}, and  underneath \eqref{eq: second contra} \BBB and \eqref{unnderneath})\EEE}. 

    \emph{Step 1: Small $p$-elastic energy on some ball.}  As the crack $K$ is Ahlfors-regular in the sense of (\ref{eqn:AhlforsReg}), we may apply Lemmas \ref{lemma: david} and \ref{lemma: cheb} to see that   there are $y \in K \cap B(x_0,r_0/3)$ and $r\in  [c_{\eps}r_0/3,r_0/3] \EEE $ such that 
    \begin{equation}\label{omegappp}
        \omega_p(y,r)  = \left(r^{p/2} \dashint_{B(y,r)} |e(u)|^p\,  {\rm d}x  \right)^{2/p} \leq \eps.
    \end{equation}
     Here, we have also used \eqref{eqn:AhlforsReg2}. \EEE Without loss of generality, we suppose that $y=0$.

    \emph{Step 2: Most balls have boundary away from the crack.}  We show that there exists   $\rho \in (r/2,3r/4)$ such that the union of \emph{circular arcs} 
    \begin{equation}\label{eqn:Zdef}
        A: = \big\{x \in \partial B(0,\rho)\colon \,  \dist(x,K)\leq \delta r\big\}
    \end{equation} has {controlled measure},  namely \EEE
    \begin{equation}\label{eqn:Zsize}
        \mathcal{H}^{1}(A)\leq c_Ar.
    \end{equation}
    Due to Ahlfors-regularity  (\ref{eqn:AhlforsReg}) of $K$, we can apply Lemma \ref{lemm: another david} which in  polar coordinates reads as
    $$\dashint_{r/2}^{3r/4}\mathcal{H}^{1}\big(\big\{z \in \partial B(0,\rho)\colon \, {\rm dist}(z,K)\leq \gamma r\big\}\big) \, \dd \rho \leq  4\bar{C} \gamma \BBB r \EEE \quad \text{{for all $\gamma \in (0,1)$}}.$$ 
    By choosing $\gamma := \delta \le   c_A/(8\bar{C})$, there is $\rho \in (r/2,3r/4)$ such that (\ref{eqn:Zsize}) holds.

    \emph{Step 3: Circular arcs of uniform length away from the crack.}
    We fix $\rho \in (r/2,3r/4)$ as in Step~2 so that (\ref{eqn:Zdef}) and (\ref{eqn:Zsize}) hold.
    {The complement of $A$ in $\partial B(0,\rho)$ is composed of circular arcs {with} distance  at least  $\delta r$ from $K$. In the following procedure, we modify $A$   to make sure that these arcs are {sufficiently large, while still being sufficiently far from each other and the crack $K$.}}   More precisely,      we can lengthen the arcs  $\partial B(0,\rho)\sm A$ \EEE by including a ball of size $r\delta/2$ around each point of $\partial B(0,\rho)\sm A$ {and} define the \emph{good set} 
    \begin{align*}
        A_{\rm g} : = \big\{x\in \partial B(0,\rho)\colon {\rm dist }(x,\partial B(0,\rho)\sm A)< \delta r/2\big\}.
    \end{align*} 
Note that by definition of $A$, each point within $A_{\rm g}$ is at least  distance $\delta r/2$ away from $K$.     The set $A_{\rm g}$ is now composed of open arcs of {length} at least $\delta r$ and in particular, it has at most $2 \pi r /(\delta r) = 2 \pi/\delta$ arcs.
    We then let $A_{\rm b} := \partial B(0,\rho)\sm A_{\rm g}$ and we note that $A_{\rm b} \subset A$. This is called the  \emph{bad set} as points in $A_{\rm b}$ are close to $K$. Since $A_{\rm g}$ is composed of {at most} $2\pi/\delta$ arcs, the same must hold for $A_{\rm b}$.

    {Next, we {modify the sets to ensure} that the arcs of $A_{\rm g}$ are not too close  to  each other.
    For this,  we  fatten each arc of $A_{\rm b}$ by $r\mu/2$ with $\mu \in (0, \delta/4]$ as chosen in \eqref{clear2} and we denote this set by  $A_{\rm b}^\mu$.}
    Thus, $A_{\rm b}^\mu$  is composed of at most $2\pi / \delta$ intervals  of length at least $\mu r$.
  Using {$A_{\rm b} \subset A$} and \eqref{eqn:Zsize} we find 
    \begin{equation}\label{eqn:ZGalpha} 
        \mathcal{H}^1(A_{\rm b}^\mu)\leq c_A r + 2\pi\mu r/\delta.
    \end{equation}
    Lastly, we let $A_{\rm g}^\mu := \partial B(0,\rho) \sm A_{\rm b}^\mu$ and we denote the union of arcs  $A_{\rm g}^\mu$ on $\partial B(0,\rho)$ by   $A_{\rm g}^\mu = \bigcup_i I_i$.
    {Let us summarize the main properties of the arcs $I_i$.
        By construction, \cl{they are relatively open and} their    geodesic distance from each other in $\partial B(0,\rho)$ is at least $\mu r$.
        By definition of $A_{\rm g}$ and the fact that $\mu \le \delta/4$, their    length is at least $\delta r/2$, and  as $A_{\rm g}^\mu \subset A_{\rm g}$, their distance from $K$ is at least $\delta r/2$.}

    \begin{figure}
        \begin{tikzpicture}[x=1cm,y=1cm]
            \draw[very thick](0,0) circle (3.0);
            \draw [red,very thick] plot [smooth] coordinates {(-3.0,0.1)  (-2.5,0.0) (-2,0.2) (0,0)};
            \draw [red,very thick] plot [smooth] coordinates {(0,0) (1,1) (1.5,1.2) (2,1.8) (2.5,1.7)};
            \draw [red,very thick] plot [smooth] coordinates {(0,0) (-0.2,-0.2) (0.5,-0.8) (1.3,-1.5) (1.8,-1.6) (1.9,-2.35)};
            \draw[dotted](0,0) circle (2.0);
            \centerarc[blue,ultra thick](0,0)(30:50:2);
        \centerarc[blue,ultra thick](0,0)(300:320:2);
    \centerarc[blue,ultra thick](0,0)(165:185:2);
\centerarc[](0,0)(50:165:2.3);
\centerarc[](0,0)(185:300:2.3);
\centerarc[](0,0)(-40:30:2.3);
\centerarc[](0,0)(50:165:1.7);
\centerarc[](0,0)(185:300:1.7);
\centerarc[](0,0)(-40:30:1.7);
\draw[rotate=30] (1.7,0) -- (2.3,0);
\draw[rotate=50] (1.7,0) -- (2.3,0);
\draw[rotate=165] (1.7,0) -- (2.3,0);
\draw[rotate=185] (1.7,0) -- (2.3,0);
\draw[rotate=300] (1.7,0) -- (2.3,0);
\draw[rotate=320] (1.7,0) -- (2.3,0);
\node[above left] at (0, 0) {$K$};
\draw[fill](0,0) circle (1pt);
\draw[->] (-3.5,0.5) -- (-2.1,0.4);
\node[left] at (-3.5,0.5) {$A_{\rm b}^\mu$};
\node[] at (2,0) {$D_1$};
\node[] at (-0.58,1.91) {$D_2$};
\node[] at (-0.93,-1.72) {$D_3$};
\end{tikzpicture}
\caption{The above picture helps to  clarify the variety of sets introduced in the proof of Theorem \ref{thm: david}.}
\label{fig:UR}
\end{figure}
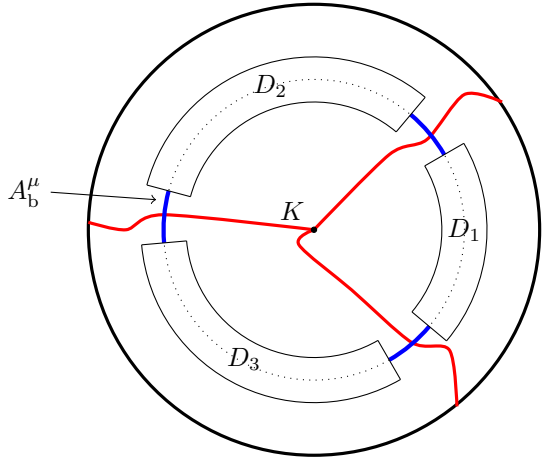

\emph{Step 4: Circular domains for Korn's inequality.} 
We now introduce domains whose size is comparable in size to the whole ball and which do not intersect $K$. On these sets, we can approximate $u$ by  rigid motions. To be precise, we define the \emph{circular domains}   
$$D_i := \big\{t\nu \colon t > 0, \ \nu \in \mathbb{S}^1 \text{ with } |t-\rho|< \delta r \EEE/4 \text{ and } \rho \nu \in I_i \big\} \subset B(0,r),$$
and note that $D_i \cap K = \emptyset$ because the points of $I_i$ are at \ks{least} distance $ \delta r /2$ from $K$.
We also observe that   the width and length of $D_i$ are at least proportional to $\delta r$,  see \EEE Figure~\ref{fig:UR} for an illustration. Recalling \eqref{omegappp},  on each $D_i$, we can apply the Korn-Poincar\'e and Korn inequality to find 
\begin{equation}\label{eqn:KornAnnuli}
    \frac{1}{r^p} \int_{D_i}|u(x) -  (A_i\, x + b_i)|^p \, \dd x + \int_{D_i}|\nabla u -  A_i|^p \, \dd x \leq C_{\rm Korn} \int_{D_i} |e(u)|^p \, \dd x \le C_{\rm Korn}  r^{2-p/2} \eps^{p/2},
\end{equation}
where $A_i \in \R^{2\times 2}_{\rm skew}$  and $b_i  \in \R^2$. Here, the constant $C_{\rm Korn}$ is independent of the exact shape of $D_i$ and only depends on $\delta$. We now show that for two distinct indices  $i$ and $j$ it holds that 
\begin{equation}\label{eqn:domainJump}
    \frac{|b_i-b_j| + r|A_i - A_j|}{\sqrt{r}} \geq c_{\rm jump}
\end{equation}
for some constant $c_{\rm jump}$ chosen small enough (depending only on $\mathbb{C}$, $\mu$, $\delta$, and $p$) in  \eqref{clear} below.
{Inequality (\ref{eqn:domainJump}) says that, in a region of the crack with low elastic energy, the rescaled displacement {$r^{-1/2} u(rx)$} has a jump bounded from below.}

\emph{Step 5: Lower bound on the jump.}
To show that (\ref{eqn:domainJump}) is true, we assume by contradiction that its converse holds for each pair $D_i$ and $D_j$, that is,
\begin{equation}\label{eqn:FalsedomainJump}
    \frac{|b_i-b_j| + r|A_i - A_j|}{\sqrt{r}} \leq c_{\rm jump}.
\end{equation}
We then construct a competitor contradicting the minimality of $(u,K)$ for a sufficiently small constant $c_{\rm jump}.$
We will modify $u$ to construct a displacement $\bar u$ without cracks on the annular ring 
\begin{equation*}
    R:= B\left(0,\rho + \delta r/4\right)\sm \overline{B\left(0,\rho - \delta r/4\right)}.
\end{equation*}
{The key point is that the quantity on the left-hand side of (\ref{eqn:FalsedomainJump}) will control the energetic cost of an interpolation between the rigid motions $A_i x + b_i$ and $A_j x + b_j$.}

\BBB We start by constructing a Sobolev function which interpolates between the rigid motions \ks{appearing} in \eqref{eqn:KornAnnuli}. Suppose without restriction  that there are at least two domains $D_i$ since the other case is even easier, as we briefly indicate below.  \EEE We consider a family of functions $(\varphi_i)_i \in C^{\infty}(R)$ such that $\sum_i \varphi_i = 1$, $0 \leq \varphi_i \leq 1$, $\varphi_i = 1$ on $D_i$ and $\abs{\nabla \varphi_i} \leq C / (\mu r)$. \BBB (Here we use that the    geodesic distance of each pair in $(I_i)_i$ is bounded from below by $\mu r$.) \EEE We can also assume that in the neighborhood of each point $x \in R$, there are at most two functions $\varphi_i$ which are non zero.
Then we set
\begin{equation*}
    \bar{u}(x) := \sum_i \varphi_i(x) a_i(x)
\end{equation*}
and we compute
\begin{equation*}
    e(\bar{u})(x) = \sum_i \nabla \varphi_{\ks{i}}(x) \odot a_i(x) = \sum_i \nabla \varphi_{\ks{i}}(x) \odot (a_i(x) - a_{i_0}(x)),
\end{equation*}
where $i_0$ is a fixed index. By assumption (\ref{eqn:FalsedomainJump}), we can estimate $\abs{e(\bar{u})} \leq C \mu^{-1} c_{\rm jump} r^{-1/2}$ in $\BBB R \EEE $, and thus
\begin{equation*}
   \BBB  \int_{R} \EEE \abs{e(\bar{u})}^2 \dd{x} \leq C_{\mu} c_{\rm jump}^2 r,
\end{equation*}
where in the following we denote by $C_{\mu}$ generic constants that \BBB depend \EEE on $\mu$ (and thus on $\delta$), and possibly also on $\mathbb{C}$ and $p$.     \BBB Clearly, the argument still works if there is only one set $D_1$.

 By Korn's inequality on $R$ and an extension result in $W^{1,2}$, we can extend $\bar{u}$ to a function in $W^{1,2}(B(0,r);\R^2)$, still denoted by $\bar{u}$, such that
\begin{equation}\label{eq_EL1}
  \int_{B(0,r)}  \abs{e(\bar{u})}^2 \dd{x} \leq C_{\mu} c_{\rm jump}^2 r.
\end{equation}
The Korn-Poincaré inequality (\ref{eqn:KornAnnuli}) allows \ks{us} to choose a radius $s$ such that $\abs{\rho - s} < \delta r/4$, $u - \bar{u} \in W^{1,p}(\partial B(0,s)\ks{\cap D};\R^2)$,    and
\begin{equation}\label{eq_el}
    \frac{1}{r^p} \int_{\partial B(0,s) \cap D} \abs{u - \bar{u}}^p \dd{x} + \int_{\partial B(0,s) \cap D} \abs{\nabla u - \nabla \bar{u}}^p \dd{x} \leq \cl{2} C_{\rm korn} r^{1-p/2} \varepsilon^{p/2}\ks{/\delta},
\end{equation}
where \BBB we set \EEE $D := \bigcup_i D_i$. Now, we make an extension of $u - \bar{u}$ from $\partial B(0,s) \cap D$ (which \ks{has positive distance from} $K$) to the whole circle $\partial B(0,s)$.
We consider for each $i$ a slightly smaller arc $\widehat{I}_i \subset I_i$ obtained by narrowing $I_i$ by a length $\mu r/10$ and define correspondingly
$$\widehat{D}_i := \big\{t\nu \colon t > 0, \ \nu \in \mathbb{S}^2 \text{ with } |t-\rho|<\delta r/4 \text{ and } \nu\in \mathbb{S}^1 \text{ with } \rho \nu \in \widehat{I}_i \big\}$$
and $\widehat{D}  := \EEE \bigcup_i \widehat{D}_i$.
As there are at most $2 \pi / \delta$ arcs in the family $(\widehat{I}_i)_i$, the complement of the arcs $\widehat{I}_i$   in $\partial B(0,\rho)$ is only slightly larger than (\ref{eqn:ZGalpha}), namely
    $\HH^1 (\partial B(0,\rho) \setminus \bigcup_i \widehat{I}_i ) \ \leq c_A r + 3 \pi \mu r/\delta$, \BBB where we use $\mu \le \delta/4$. \EEE
Since $\partial B(0,s) \setminus \widehat{D}$ is \ks{equal to}  $\partial B(0,\rho) \setminus \bigcup_i \widehat{I}_i$ \ks{scaled by a constant smaller than $2$}, we also have
\begin{equation}\label{eq_ES}
    \HH^1(\partial B(0,s) \setminus \widehat{D}) \leq 2 c_A r + 6 \pi \mu r/\delta.
\end{equation}
Then we let $\varphi \in C^1(\partial B(0,s))$ be such that $0 \leq \varphi \leq 1$,
\begin{equation*}
    \text{$\varphi = 1$ on $\partial B(0,s) \cap \widehat{D}$} \quad \text{and} \quad \text{$\varphi = 0$ on $\partial B(0,s) \setminus D$}
\end{equation*}
and $\abs{\nabla \varphi} \leq C / (\mu r)$.
Hence, the function $f := \varphi (u - \bar{u})$ belongs to $W^{1,p}(\partial B(0,s);\R^2)$, and since
\begin{equation*}
    \abs{\nabla f} \leq \abs{\varphi} \abs{\nabla u - \nabla \bar{u}} + \abs{\nabla \varphi} \abs{u - \bar{u}},
\end{equation*}
it follows from (\ref{eq_el}) that
\begin{equation*}
    \int_{\partial B(0,s)} \abs{\nabla f}^p \dd{x} \leq C_{\mu} r^{1-p/2} \varepsilon^{p/2}.
\end{equation*}
By \cite[Lemma 22.16]{David}, there exists a function $v \in \BBB W^{1,2} \EEE (B(0,s);\R^2)$ which coincides with $f$ on $\partial B(0,s)$ such that
\begin{equation}\label{eq_EL2}
    \int_{B(0,s)} \abs{\nabla v}^2 \dd{x} \leq C r^{2-2/p} \left(\int_{\partial B(0,s)} \abs{ \nabla \EEE f}^p \dd{x}\right)^{2/p} \leq C_{\mu} r \varepsilon.
\end{equation}
We can finally define $(\hat{u},\widehat{K})$ as
\begin{equation*}
    \hat{u}(x) = 
    \begin{cases}
        v(x) + \bar{u}(x) &\text{in} \ B(0,s)\\
        u(x)              &\text{in} \ \Omega \setminus B(0,s)
    \end{cases}
\end{equation*}
and
\begin{equation*}
    \widehat{K} = (K \setminus B(0,s)) \cup (\partial B(0,s) \setminus \widehat{D}).
\end{equation*}
From (\ref{eq_EL1}) and (\ref{eq_EL2}), we estimate
\begin{equation*}
    \int_{B(0,s)} \abs{e(\hat{u})}^2 \dd{x} \leq C_{\mu} (c_{\rm jump}^2 + \varepsilon) r.
\end{equation*}
Using the almost-minimality property \eqref{amin}  to compare $(u,K)$ and $(\hat{u},\widehat{K})$ in $B(0,r)$, we find \BBB by \eqref{eq_ES} \EEE
\begin{equation}\label{clear0}
    \HH^1\big( K \EEE \cap B(0,s)\big) \leq \big(2c_A + 6 \pi \mu/\delta + C_{\mu}  (c_{\rm jump}^2 + \varepsilon ) + h(r) \big) r.
\end{equation}
Then, recalling {\eqref{eq:h000} and} \eqref{clear2}, we can choose  $\eps$ and $c_{\rm jump}$ small enough depending only on $\mathbb{C}$, $\mu$, $\delta$,  $p$, and $C_{\rm Ahlf}$ (which itself depends only on $\mathbb{C}$)  such that {we have}  
    \begin{equation}\label{clear}
        2c_A + 6 \pi \mu/\delta +  C_{\mu} ( \eps + c^{2}_{\rm jump})  + h(r) \leq \frac{1}{8}C_{\rm Ahlf}^{-1}.
    \end{equation}
    Recalling that $s \geq r/4,$ we apply \eqref{clear0}--\eqref{clear} and (\ref{eqn:AhlforsReg}) to find the contradiction $\frac{1}{4}C_{\rm Ahlf}^{-1}r \leq \frac{1}{8}C_{\rm Ahlf}^{-1} r. $ This contradiction immediately rules out the case of a single set $D_1$ \EEE \cl{and shows that} we  must have that (\ref{eqn:domainJump}) holds for a sufficiently small constant $c_{\rm jump}$ depending on  $p$,  $C_{\rm Ahlf}$,  and $\mathbb{C}$. 
    In the next steps, the constant $c_{\rm jump}$ is fixed, but $\varepsilon$  will still be chosen {sufficiently} small.

    \emph{Step 6: Constructing a partition via piecewise Korn.}
    In the above steps, we have shown a separation of values in the sense of (\ref{eqn:domainJump}) and may, up to relabeling, take $D_i$ and $D_j$ in (\ref{eqn:domainJump}) to be $D_1$ and $D_2$. In this second part, we show that the crack nearly separates $D_1$ and $D_2$,   thereby completing the \BBB proof. \EEE

    In the rest of the proof, we work   in the unit ball  $B(0,1)$. \EEE Precisely, rescaling in $B(0,1)$ via $u_{r}(x) : = r^{-1/2} u(rx)$ and $K_r := r^{-1} K$, we  let \EEE $D_{r,1}$ and $D_{r,2}$  be \EEE circular domains  coming from a change of variables with $\mathcal{L}^2(D_{r,i})\ge c_{\rm size}$, $i=1,2$, with a constant depending on $\delta $   and thus  on $C_{\rm Ahlf}$.  The Korn inequalities (\ref{eqn:KornAnnuli}) hold with $A_i$, $b_i$, and $D_i$ replaced by their appropriately rescaled versions $A_{r,i} : = \sqrt{r}A_i$, $b_{r,i} : = b_i/\sqrt{r}$, and $D_{r,i}$. Our goal is to find a    compact, connected set $ \Gamma_r$ satisfying
    \begin{align}\label{eq: to find}
        \mathcal{H}^1(\Gamma_r) \le \bar{C}_{\rm ur} \quad \quad  \text{and} \quad \quad \mathcal{H}^1\big(K_r \cap \Gamma_r \cap B(0,1)   \big) \ge \frac{1}{\bar{C}_{\rm ur}}  
    \end{align}
    for some $\bar{C}_{\rm ur} \geq 1$ only depending on $C_{\rm Ahlf}$ (and thus on $\mathbb{C}$).  Then, we conclude the property stated in Theorem \ref{thm: david} on the ball $B(x_0,r_0)$ by setting $\Gamma = r \Gamma_r$,  recalling   \cl{$r \in [c_\eps r_0/3,r_0/3]$} \EEE (see before \eqref{omegappp}), and using  the fact that $c_\eps$ only depends on $C_{\rm Ahlf}$, $p$, and $\mathbb{C}$, and therefore only on $\mathbb{C}$. 

    From now on, for convenience and with abuse  of notation, we drop the subscript $r$ from all quantities. In particular, we have that both inequalities  (\ref{eqn:KornAnnuli}) and (\ref{eqn:domainJump}) hold with $r = 1.$ In the unit ball, we apply the \BBB piecewise Korn \EEE  inequality (Proposition \ref{thm:refinedKorn}) on the set $B(0,1)$ (with $J_u = K$)  to find $u^\theta$ such that
    \begin{align}\label{along with}
        {\mathcal{L}^2(\{u\neq u^\theta\}) \leq {C} \theta  }
    \end{align}
    and a finite Caccioppoli partition $(P_j)_{j=0}^{J}$ (with $P_0 = \{u \neq u^\theta\}$), such that there are rigid motions $\tilde a_j(x) = \tilde A_j x + \tilde b_j$ satisfying
    \begin{align}
        \sum\nolimits_{j=0}^J \mathcal{H}^1\big(\partial^* P_j \cap (B(0,1) \sm K)\big) &\leq {C} \theta, \label{eqn:partitionAwayFromK} \\
        \hspace{-0.2cm}\| u^\theta  - \sum\nolimits_{j=0}^J \tilde a_j \chi_{P_j}\|_{L^\infty(B(0,1))} + \|\nabla u^\theta - \sum\nolimits_{j=0}^J \tilde A_j \chi_{P_j}\|_{L^1(B(0,1))}&\leq C_\theta \|e(u)\|_{L^p(B(0,1))} \leq C_\theta \sqrt{\eps} \label{eqn:piecewiseKorn},
    \end{align}
    where $\eps$ comes from \eqref{omegappp} and the rescaling of $u$, ${C>0}$ is a constant depending only on the  Ahlfors constant $C_{\rm Ahlf}$, and $C_\theta$  depends on $\theta$ with $C_\theta \to \infty$ as $\theta \to 0$.   We can choose $\theta(\eps)$ depending on $\eps$ with $\theta(\eps) \to 0$ as $\eps \to 0$ and {$C_{\theta(\eps)} \eps^{1/2} \leq \eps^{1/4}$.}

    We now show that for sufficiently small $\eps$ and each $i = 1,2$ there is \BBB an \EEE index $j_i \in \lbrace 1,\ldots, J \rbrace$ such that
    \begin{align}\label{eq: covers a large}
        \mathcal{L}^2\big(D_i \cap P_{j_i}\big)\geq  \mathcal{L}^2\big(D_i\big)/2  \ge  c_{\rm size}/2.
    \end{align}
    \BBB where we recall $\mathcal{L}^2\big(D_i\big)  \ge  c_{\rm size}$, as stated before \eqref{eq: to find}. \EEE     
    We assume to the contrary, {i.e.}, $\mathcal{L}^2\big(D_i \cap P_j\big)  <   \mathcal{L}^2\big(D_i\big)/2$ for all \BBB  $j \in \lbrace 0,\ldots,J\rbrace$ \EEE and we deduce from the relative isoperimetric inequality in $D_i$ (which is bi-Lipschitz equivalent to a ball) that, for each $P_j$ in the partition, we have
    $$\mathcal{L}^2(P_j \cap D_i)^{1/2}\leq C_{\rm iso}\mathcal{H}^1(\partial^*P_j \cap D_i), $$
    where $C_{\rm iso}$ depends only on the dimension and the bilipschitz equivalence (and therefore {the} Ahlfors density constant). Putting these inequalities together, we find by \eqref{eqn:partitionAwayFromK} and $D_i \cap K = \emptyset$ that
    $$\mathcal{L}^2(D_i)^{1/2}= \mathcal{L}^2\big(\bigcup\nolimits_j P_j \cap D_i\big)^{1/2} \le \sum\nolimits_j \mathcal{L}^2(P_j \cap D_i)^{1/2} \leq C_{\rm iso}\sum\nolimits_j\mathcal{H}^1(\partial^*P_j \cap D_i)\leq C_{\rm iso} \theta(\eps). $$ 
    For sufficiently small $\eps$, this is a contradiction to $ \mathcal{L}^2(D_i) \ge c_{\rm size}$. \BBB We thus conclude \eqref{eq: covers a large}, first for some index  $j_i \in \lbrace 0,\ldots,J\rbrace$, and then excluding $j_i=0$ by \eqref{along with}. \EEE

    We now show that $j_1 \neq j_2.$ Note that $j_i \neq 0$, so $u = u^\theta$ on the elements $P_{j_i}$ of the partition. Suppose by way of contradiction that $j = j_1 = j_2.$
    Precisely, by (\ref{eqn:piecewiseKorn}), \eqref{eq: covers a large},  and $ C_{\theta(\eps)} \eps^{1/2} \BBB \le \EEE \eps^{1/4}$, we have for $i=1,2$ that 
    \begin{align*}
        \eps^{1/4}& \geq \|\nabla u - \tilde A_j \|_{L^1(D_i \cap P_j)}   \geq \|\tilde A_j - A_i\|_{L^{1}(D_i \cap P_j)} - \|\nabla u -A_i\|_{L^1(D_i)}  \geq \frac{c_{\rm size}}{2} |\tilde A_j - A_i| - C_{\rm Korn} \eps^{1/2},
    \end{align*}
    where we have also used (\ref{eqn:KornAnnuli}) \cl{(rescaled to the unit ball)} with H\"older's inequality.  Rearranging,   we have \EEE 
    \begin{equation}\label{eqn:contraRgap}
        |A_1- A_2| \leq |A_1 - \tilde A_j | + |\tilde A_j - A_2|\leq C_* \eps^{1/4},
    \end{equation}
    where $C_*$ depends only on $c_{\rm size}$ and $C_{\rm Korn}$, and thus only  on \EEE $C_{\rm Ahlf}$. We now control the gap between $b_1$ and $b_2$ to contradict (\ref{eqn:domainJump}) for sufficiently small $\eps$. Up to a constant (depending as usual on the Ahlfors density), we apply the previous two estimates to (\ref{eqn:piecewiseKorn}) and use (\ref{eqn:KornAnnuli}) (again with H\"older's inequality) as well as \eqref{eq: covers a large}  to find for $i=1,2$
    \begin{align}\label{eq: second contra}
        \eps^{1/4} &\geq \|u - \tilde a_j\|_{L^1(D_i \cap P_j)}  \geq \|b_i - \tilde b_j\|_{L^1(D_i \cap P_j)} - C_*\eps^{1/4} - \| u - (A_i\, \cdot + b_i)\|_{L^1(D_i\cap P_j)}\notag \\
                   & \geq \frac{c_{\rm size}}{2}|b_i - \tilde b_j| - C_*\eps^{1/4} - C_{\rm Korn} \eps^{1/2}.
    \end{align}
    Increasing $C_*$, as before, we have
    $|b_1 - b_2 | \leq C_*\eps^{1/4}.$
    For small enough $\eps$, this inequality and (\ref{eqn:contraRgap}) contradict (\ref{eqn:domainJump}).  Therefore, it holds that \EEE   $j_1 \neq j_2$.

    \emph{Step 7: The crack almost separates the circular domains.}
    We now form a new partition by including all of $D_i$ in $P_{j_i}$. For this, we introduce $$P_{i,j} := D_i \cap P_{j} \quad \text{for all $j=0,\ldots,J$ {and $i=1,2$}}.$$ 
    We note that for each  $j \neq j_i$ {we have $P_j \cap P_{j_i} = \emptyset$,} so $\mathcal{L}^2(D_i \cap P_{j})\le   \mathcal{L}^2(D_i)/2$ by \eqref{eq: covers a large}, and then Lemma~\ref{lemma: maggi} yields, up to a constant $\hat{C}_1$ depending  on the Lipschitz constant of $D_i$ (and thus depending only on $\delta$), \EEE  
    \begin{equation}\label{eqn:Pijcontrol}
        \mathcal{H}^1(\partial^* P_{i,j}) \leq \hat{C}_1 \mathcal{H}^1(\partial^* P_j \cap D_i) = \hat{C}_1 \mathcal{H}^1\big(\partial^* P_j \cap (D_i \setminus K)\big),
    \end{equation}
     where \EEE the second step follows from  $K \cap D_i = \emptyset$. We define a new partition  $(\tilde{P}_j)_{j=0}^{J}$  by distinguishing three cases. {We define the two} sets
    $$\tilde P_{j_1} : = (P_{j_1} \setminus D_2) \cup \bigcup\nolimits_{j\neq j_1} P_{1,j},\quad \quad \quad \quad \tilde P_{j_2} : = (P_{j_2} \setminus D_1) \cup \bigcup\nolimits_{j\neq j_2} P_{2,j}, $$
    and  for \EEE {the remaining} $j \ne j_1,j_2$ we define
    $$\tilde P_j = P_j\sm (\tilde P_{j_1} \cup \tilde P_{j_2}) .$$
     Without loss of generality, we may suppose $\mathcal{L}^2(\tilde P_{j_1}) \leq \frac{1}{2}\mathcal{L}^2(B(0,1))$, otherwise we can use $\tilde P_{j_2}$ in its place. \EEE  
    {Notice that $\tilde P_{j_1}\supset D_1$ and that, by (\ref{eqn:partitionAwayFromK}), (\ref{eqn:Pijcontrol}) and $j_1 \ne j_2$,  we find 
        \begin{align}\label{eq: j1}
            \HH^1\big(\partial^* \tilde{P}_{j_1} \cap (B(0,1) \setminus K)\big) &\leq \sum_{j \ne j_1} \HH^1(\partial^* P_{1,j}) + \HH^1\big(\partial^* P_{j_1} \cap (B(0,1) \setminus K) \big) + \HH^1(\partial^* (P_{j_1} \cap D_2))  \notag \\
                                                                  &  \leq \hat{C}_1    \sum\nolimits_{j} \HH^1\big(\partial^* \BBB {P}_{j} \EEE \cap (B(0,1) \setminus K)\big) \leq \hat{C_1} \theta(\varepsilon)
        \end{align}
    for {a} bigger constant  $\hat{C_1} \geq 1$, \BBB where we used that $P_{j_1} \cap D_2 = P_{2,j_1}$.\EEE} 
    Note that $\tilde P_{j_1}$ can be written as the disjoint union of indecomposable sets of finite perimeter, and only one of these contains $D_1$. Denote this component by $Q$. As $Q$ is indecomposable, up to saturating the set {\cite[Definition 5.2]{Ambrosio-Morel},} {its} reduced boundary is essentially closed and connected (precisely, a Jordan curve {by the structure theorem of the boundary of planar sets of finite perimeter, see  \cite[Corollary 1]{Ambrosio-Morel} and}  Section~\ref{korn-sec}). \EEE 
    As \cl{$\mathcal{L}^2(\tilde P_{j_1}) \leq \frac{1}{2}\mathcal{L}^2(B(0,1))$ and} $\overline{\partial^* Q} $ necessarily separates $D_1$ from $D_2$ in $B(0,1)$, we have that $\mathcal{H}^1(\partial^* Q \cap B(0,1))\geq c_{\rm dim},$ where  $c_{\rm dim}$ depends on the diameter of $D_1$ \BBB (cf.\ Lemma \ref{lemma: diam} \ks{ and \ref{lemma: maggi}}) \EEE and thus only  on \EEE  $C_{\rm Ahlf}$. Consequently, as $\partial^* Q \subset \partial^* \tilde P_{j_1}$ up to a null set,   by \eqref{eq: j1} we get
\begin{align}\label{unnderneath}
\mathcal{H}^1\big((\partial^* Q \cap K) \cap B(0,1)\big) \geq \mathcal{H}^1\big( \partial^* Q \cap B(0,1)\big) - \mathcal{H}^1\big((\partial^* Q \sm K) \cap B(0,1)\big) \geq c_{\rm dim} - \hat{C}_1 \theta(\eps).
\end{align}
    Clearly, {{as} $\partial^* Q \subset \partial^* \tilde P_{j_1}$}, by \eqref{eq: j1} and \eqref{eqn:AhlforsReg} for $K$ we also find $\mathcal{H}^1(\partial^* Q \cap B(0,1))\le \hat{C}_2$ for $\hat{C}_2$ depending on $C_{\rm Ahlf}$. Choosing $\eps$ sufficiently small,  we see that the  compact, connected set $ \Gamma :=  \overline{\partial^* Q}\EEE$ satisfies \eqref{eq: to find} (recall that we dropped \BBB the \EEE subscript $r$), which concludes the proof.
\end{proof}

We close this section with an important consequence of local uniform rectifiability: it allows to initialize $\beta$ in many balls  which will be instrumental for the porosity property proved in Section \ref{sec: main result} below.

According to \cite[Theorem 41.3]{David}, if $\Gamma \subset \R^2$ is an Ahlfors-regular curve with constant $C > 0$ (recall \eqref{eq_z1} and \eqref{eq_z2}), there exists a constant $C_0 \geq 1$ (depending on the aforementioned $C$) such that, for all $x \in \Gamma$ and $r > 0$,
\begin{equation*}
    \int_{\Gamma \cap B(x,r)} \Big( \int_0^r \beta_{\Gamma}^2(y,t)  \frac{\dd{t}}{t} \Big)  \dd{\HH^1(y)} \leq C_0 r.
\end{equation*}
Let us fix $x \in K$ and $r>0$ such that $B({x},2r) \subset \Omega$   and $h({2}r) \leq \varepsilon_{\rm ur}$.    {Then by Theorem \ref{thm:UniformRect} and the definition of local uniform rectifiability in \eqref{eq_defi_ur},} we get that $K \cap B(x,r)$ is contained in an Ahlfors-regular curve {$\Gamma$} with constant ${C_{\rm ur}} > 0$. Therefore, {as $\beta_K\leq \beta_\Gamma$ for balls contained in $B(x,r)$,} there  exists a constant $C_0 \geq 1$ {(depending on {$C_{\rm ur}$})} such that  
\begin{equation}\label{eq: 5.28}
    \int_{K \cap B(x,{r/2})} \Big( \int_0^{{r/2}} \beta_{K}^2(y,t)  \frac{\dd{t}}{t} \Big) \dd{\HH^1(y)} \leq C_0 r.
\end{equation}
Given $\eps>0$, in view of Lemma \ref{lemma: cheb}, this allows us to find  $y \in K \cap B(x, r/2) \EEE$ and $t>c_{\eps}  r/2 \EEE $ such that  $\beta_{K}^2(y,t)\le \eps$. This will be a crucial property in the next section.

\section{Dimension of the singular set of  \texorpdfstring{$K$}{K}: Proof of  Theorem \texorpdfstring{\ref{th: main result}}{2.4}\EEE}\label{sec: main result}

This section is devoted to the proof of Theorem \ref{th: main result}. The proof fundamentally relies on a porosity property which we discuss in Subsection \ref{sec: porosity}. The latter also requires a suitable initialization of the jump whose detailed treatment is postponed to Subsection \ref{sec: look at jumo}.

\subsection{Porosity and proof of the statement} \label{sec: porosity}

We investigate the dimension of the set $K^*$ of points $x \in K$ where $K$ is nonsmooth. We are going to show that $K^*$ has a dimension less than $1$ relying on the fact that $K^*$ is a porous set, i.e., it has many holes in a quantified way. In other words, we will prove that $K$ is smooth in many places. {This property will also imply the higher integrability of the gradient, following the argument of \cite{DPF}.}

Our approach {for} porosity is close to {\sc Rigot} \cite{Rigot} (see also \cite[Section 51]{David}), yet with some important differences. Rigot's approach relies on the standard Mumford--Shah epsilon-regularity theorem which requires $ \beta \EEE + \omega_2$ to be small in order for $K$ to be smooth.
In this case, the porosity is proved by finding many balls where $ \beta \EEE  + \omega_2 \leq \varepsilon$.
The procedure is in two steps:  First, one finds many balls where $ \beta \EEE  + \omega_p \leq \varepsilon$, where $p$ is some fixed exponent in $(1,2)$. Next, one proves a kind of reverse Hölder inequality (with additional error terms) to initialize $\omega_2$ using $\beta + \omega_p \leq \varepsilon$.
In contrast,  our epsilon-regularity theorem requires {us} to control $\beta + J^{-1}$. We thus  have the same first step but we then use  the lemma below to  initialize $J^{-1}$ using $\beta + \omega_p \leq \varepsilon$.

\begin{lemma}[Initializing $J$]\label{lem_init_J}
    Let $p \in (4/3,2]$ and $(u,K)$ be a Griffith almost-minimizer \BBB  with respect to a gauge $h$. \EEE Let $x_0 \in K$ and $r_0 > 0$ be such that $B(x_0,r_0) \subset \Omega$.
    For all $C_0 > 0$, there exists $\varepsilon_0 \in (0,1/2]$ and $\gamma \in (0,1/2)$,  both {depending} only on $C_0$ and $\mathbb{C}$,  such that, if
    \begin{equation*}
        \beta(x_0,r_0) +  \omega_p(x_0,r_0) + h(r_0) \leq \varepsilon_0,
    \end{equation*}
    then there exists $y \in B(x_0,r_0/2) \cap K$  with
    \begin{equation}\label{CCC0}
        {J}(y,\gamma r_0) \geq C_0.
    \end{equation}
\end{lemma}

{We defer the proof of the above lemma to the following subsection. We remark that the initalization of a similar quantity in the case of the Mumford--Shah functional, which penalizes {translations}, would be relatively straightforward. In contrast, we must explicitly use the rigidity of the affine maps appearing in the jump \eqref{def:normalizedJump} (rigid in the sense that  for skew-affine $A$ and any $v \in \R^2$, one has $|Av| = \frac{1}{\sqrt{2}}|A||v|$)  to overcome the loss of control resulting from frame indifference.}

With this at hand, we can formulate the porosity property. 

\begin{proposition}[Porosity]\label{prop_porosity}
    For each choice of exponent $\alpha \in (0,1)$, there exist  constants $\varepsilon_{por} \in {(}0,1/2]$ and $\gamma_{\rm por} \in (0, 1/2) \EEE $ (depending on $\mathbb{C}$ and $\alpha$) such that the following holds. {Let $(u,K)$ be an}   almost-minimizer  for the Griffith energy \eqref{eq: main energy} with gauge $h(t) = h(1) t^\alpha$.  {For} all $x \in K$ and $r > 0$ with $B(x,r) \subset \Omega$ and $h(r) \leq \varepsilon_{por}$, there exists $y \in {K \cap B(x,r/2)}$ and $t \in [\gamma_{\rm por} r,r/2]$ such that
    \begin{equation*}
        \text{$K$ is $C^{1,\alpha/2}$ in $B(y,t)$}.
    \end{equation*}
\end{proposition}

\begin{proof}
    We choose  $\eps_{\rm por}    =  \min(\varepsilon_{\rm Ahlf},\varepsilon_{\rm ur}, \eps_0/2, \eps_*)$,  where $\varepsilon_{\rm Ahlf}$ is the Ahlfors-regularity parameter {(see \eqref{eqn:AhlforsReg})}, $\varepsilon_{\rm ur}$ is the uniform rectifiability parameter {(see Theorem \ref{thm:UniformRect} and the associated definition)},  $\eps_0$ is the {epsilon-regularity constant} in \eqref{eq: smallli epsi}, and $\eps_*$ is a constant only depending on $\eps_0$ and $\mathbb{C}$ chosen in Step~1 below.
    {Note that $\eps_{\rm por}$ only depends on $\mathbb{C}$ and $\alpha$.}
    Our goal is to prove that there exists $\bar{\gamma}_{\rm por}  \in (0, 1/2) \EEE$, depending on $\mathbb{C}$ and $\alpha$, such that for all $x \in K$, $r > 0$ with $B(x,r) \subset \Omega$, and $h(r) \leq \eps_{\rm por} $, there exists $y \in {K \cap B(x,r/2)}$ and $t \in [\bar{\gamma}_{\rm por}  r,r/2]$ such that 
    \begin{equation}\label{eq: toooooprpp}
        \beta(y,t) + J^{-1} (y,t)\leq \varepsilon_0/2. 
    \end{equation}
    {Once we have {proven} \eqref{eq: toooooprpp}, in view of  $h(r) \leq  \eps_{\rm por} \le \eps_0/2$, we can apply epsilon-regularity (Theorem \ref{th: eps_reg}) in the ball $B(y,t)$. We conclude that there exists a constant $\gamma \in (0,1)$, depending on $\mathbb{C}$ and $\alpha$, such that $K$ is $C^{1,\alpha/2}$ in $B(y,\gamma t)$.
    This {concludes the proof of the proposition with} the choice $\gamma_{\rm por} := \gamma \bar{\gamma}_{\rm por}$.}
    {It remains to} show \eqref{eq: toooooprpp}.   

    \emph{Step 1: Initializing $J$ given control on $\beta$ and $\omega_p$.}   Let us fix an exponent $p \in (4/3,2)$.
    According to Lemma \ref{lem_init_J}, there {exist constants} $\eps_*  > 0$ and $\gamma_* \in (0,1/2)$  depending on  $\varepsilon_0$ and $\mathbb{C}$  such that {for all} $z \in K$ and $s > 0$ with $B(z,s) \subset \Omega$   the {condition}
    \begin{equation}\label{eq_bov}
        \beta(z,s) + \omega_p(z,s)  + h(s) \leq 3  \eps_*  
    \end{equation}
    {implies} the existence of $y \in K \cap B(z,s/2)$ such that
    \begin{equation*}
        J^{-1}(y, \gamma_* s) \leq  \varepsilon_0/4. 
    \end{equation*}
    According to the scaling properties of $\beta$ (Remark \ref{rmk_beta}), we also have {$\beta(y,\gamma_* s) \le 2\gamma_*^{-1} \beta(z,s) \leq 6 \gamma_*^{-1} \varepsilon_*$}. Thus, if we take $\eps_* $ even smaller, still depending only on  $\varepsilon_0$  and $\mathbb{C}$, \EEE we obtain
    \begin{equation*}
        {\beta(y,t) + J^{-1}(y,t) \leq  \varepsilon_0/2, \quad \text{where $t = \gamma_* s$.}}
    \end{equation*}
\BBB To complete the proof in Step 3 below, it basically remains  \EEE to  justify the initialization of (\ref{eq_bov}).

    \emph{Step 2: Initializing $\beta$ and $\omega_p$.} 
    Let $\eps_* > 0$ as in Step 1 and  recall that    it only depends on $\varepsilon_0$ and $\mathbb{C}$. \EEE We show that there exists $\gamma_0 \in (0, 1/4) \EEE$ depending on   $\eps_* $ {and $\mathbb{C}$} such that for all $x \in K$ and $r > 0$ with $B(x,r) \subset \Omega$ and $h(r) \leq \eps_{\rm por}  $, there {exist} $z \in {K} \cap B(x,r/4)$  and $s \in [\gamma_0 r, r/4]  $ such that
    \begin{equation*}
        {  \beta(z,s)  \leq \eps_* , \quad \quad   \omega_p(z,s)  \le \eps_* .}
    \end{equation*}
    To this end,  let  $\bar{\eps} \le \eps_* $ to be specified below in \eqref{aaaa choose}.  We can apply Theorem \ref{thm:UniformRect} (for {$r/2$} in place of $r$)  to find that $K \cap B(x, {r/2})$ is contained in an Ahlfors-regular curve. Then, as explained in \eqref{eq: 5.28},   we have
    $$ \int_{K \cap B(x,r/8)} \Big( \int_0^{r/8} \beta^2(y,t)  \frac{\dd{t}}{t} \Big) \dd{\HH^1(y)} \leq C_0 r
    $$
    for a constant $C_0 \geq 1$ which depends on $\mathbb{C}$. Therefore,  by  applying Lemma \ref{lemma: cheb} for $\bar{\eps}^2$, we find $\gamma_1 \in (0,1/8)$ depending only on $\bar{\eps}$ {and $\mathbb{C}$} as well as $y \in K \cap B(x,r/8)$ and $t \in [\gamma_1 r, r/8]$ such that      
    \begin{equation}\label{eq_porosity_beta}
        \beta^2(y,t) \leq \bar{\varepsilon}^2.
    \end{equation}
    {Next, we apply Lemma \ref{lemma: david} \BBB (using also \eqref{eqn:AhlforsReg2}) \EEE to find
        $$\int_{K \cap B(y,t/8)} \int_0^{t/8} \omega_p(z,s) \frac{\dd{s}}{s} \dd{\HH^1(z)} \leq C_p t,$$
        for a constant $C_p \geq 1$ which depends on $\mathbb{C}$.  Then,  we apply Lemma~\ref{lemma: cheb} for $\eps_* $} and find       some constant $\gamma_2 \in (0,1/8)$ depending   $\eps_* $ and $\mathbb{C}$ as well as some $z \in {K} \cap B(y,t/8)$ and $s \in [\gamma_2 t,t/8]$ such that
    \begin{equation*}
        \omega_p(z,s) \leq \eps_* .
    \end{equation*}
    By (\ref{eq_porosity_beta}), Remark \ref{rmk_beta}, and $s \geq \gamma_2 t$ we also have $\beta(z,s) \leq 2 \gamma_2^{-1} \bar{\varepsilon}$. Thus, choosing $\bar{\eps}$ small enough depending on $\gamma_2$ and $\eps_* $ (and thus depending only on $\eps_* $ {and $\mathbb{C}$}), we get
    \begin{equation}\label{aaaa choose}
        \beta(z,s)   \leq 2 \gamma_2^{-1}   \bar{\varepsilon} \le \eps_*.
    \end{equation}
    This shows Step 2 for the choice $\gamma_0 := \gamma_1 \gamma_2$.

    \emph{Step 3: Conclusion.}
    For all $x \in K$ and $r > 0$ with  $B(x,r) \subset \Omega$   and $h(r) \leq \varepsilon_{\rm por} $,   we apply Step 2 to find a smaller ball $B(z,s)$  in $B(x,r)$ with   $z \in K \cap B(x,r/4)$ and  $s \in [\gamma_0 r , r/4]$ such that $\beta(z,s) + \omega_p(z,s) \leq 2  \eps_* $. Then, as $h(r) \leq  \varepsilon_{\rm por} \le \eps_*$,  we apply Step 1  to find a smaller ball $B(y,t)$, with  $y \in K \cap   B(z,s/2) \subset   B(x,r/2)$ and  $t = \gamma_* s$, such that  $\beta(y,t) + J^{-1}(y,t) \leq \varepsilon_0/2$. This  shows \eqref{eq: toooooprpp} and thus \EEE the statement for the choice $\bar{\gamma}_{\rm por}   = \BBB \gamma_* \EEE \gamma_0$. 
\end{proof}

Theorem \ref{prop_porosity} directly  implies Theorem \ref{th: main result}(i) by a standard abstract argument.

\begin{proof}[Proof of  Theorem \ref{th: main result}(i)]
    {By either} \cite[Theorem 51.20, Lemma 51.21]{David} or \cite[Lemma 5.8]{DaSe}, the porosity property implies the estimate on the Hausdorff dimension of $K^*$.
\end{proof}

We finally pass to Theorem \ref{th: main result}(ii), i.e., to  the higher integrability of the gradient in the case $h \equiv 0$. Here, we  follow  the technique of \cite{DPF} which applies to our case,  even though we use the symmetric gradient $e(u)$ instead of the full gradient $\nabla u$. More precisely, we can apply \cite[Lemma 7.1]{CL2}  which follows from the technique of \cite{DPF}. The lemma is stated below in dimension two for convenience of the reader. Here, we say that $K$ is $C^{1,\alpha}$-regular in a ball $B(x_0,r)$ with $x_0 \in K$ if $K \cap B(x_0,r)$ is the graph of a $C^{1,\alpha}$ function $f$ such that, in a coordinate system, it holds $f(0) = 0$, $f'(0) = 0$ and $r^\alpha \| f' \|_{C^\alpha} \leq 1/16$.
{This constraint on the normalized Hölder norm plays a role in the proof of \cite[Lemma 7.1]{CL2}.  Yet, we emphasize that,  if we have $r^{\alpha} \Vert f' \Vert_{C^{\alpha}} \leq C$ for some universal constant $C \geq 1$, then we can find a smaller radius $t \in (0,r)$, with $r/t$ only depending on $C$, such that $t^{\alpha} \Vert f' \Vert_{C^{\alpha}} \leq 1/16$. Now, inspection of the proof of epsilon-regularity, see \cite[Proposition 4.1]{FLSeps}, shows that  the graph in the statement of  Theorem \ref{th: eps_reg} satisfies $r^{\alpha} \Vert f' \Vert_{C^{\alpha}} \leq C$.  From the above comments, we can thus assume that the graph is $C^{1,\alpha}$-regular.

\begin{lemma}[Lemma 7.1 of \cite{CL2}]\label{lem_DPF}
    We fix a radius $R > 0$.
    Let $K$ be a closed subset of $B(0,R) \subset \R^2$ and $v \colon B(0,R) \to [0,\infty)$ be a nonnegative Borel function.
    We assume that there exists $C_0 \geq 1$ and $0 < \alpha \leq 1$ such that the following holds true:
    \begin{enumerate}
        \item For each ball $B(x,r) \subset B(0,R)$ with $x \in K$, it holds that
            \begin{equation}\nonumber
                C_0^{-1} r \leq \HH^{1}(K \cap B(x,r)) \leq C_0 r.
            \end{equation}
        \item For each ball $B(x,r) \subset B(0,R)$ with  $x \in K$, there exists a smaller ball $B(y,C_0^{-1} r) \subset B(x,r)$ with $y \in K$ in which $K$ is $C^{1,\alpha}$-regular.
        \item For each ball $B(x,r) \subset B(0,R)$ such that either   $K \cap B(x,r) = \emptyset$ or such that $x \in K $ and $K$ is $C^{1,\alpha}$-regular in $B(x,r)$, we have
            \begin{equation}\label{ass:3}
                \sup_{y \in B(x,r/2)} v(y) \leq C_0 r^{-1}.
            \end{equation}
    \end{enumerate}
    Then, there exists $q > 1$  and $C \geq 1$ depending on $C_0$ such that
    \begin{equation}\nonumber
        \int_{B(0,R/2)} \! v^q  \, \dd x \leq C R^{2-q}.
    \end{equation}
\end{lemma}

\begin{proof}[Proof of  Theorem \ref{th: main result}(ii)]
    Property \eqref{eq: almost last} for all $x \in \Omega$ and $r > 0$ such that $B(x,\BBB 2r) \EEE \subset \Omega$ will follow from  Lemma \ref{lem_DPF} applied in the ball \cl{$B(x,2r)$} to the function $v := \abs{e(u)}^2$, with $p = 2q$,  once we have checked all the assumptions.

    Assumption (1) follows from the  Ahlfors-regularity of $K$, see \eqref{eqn:AhlforsReg}.  

    Assumption (2) follows from Proposition \ref{prop_porosity}. (Since there is no gauge, we indeed get that $K$ is $C^{1,1/4}$ in the smaller ball  with a small {(normalized)} H\"older constant, \EEE see \cite[Remark 2.5(i)]{FLSeps} {and the comment just before Lemma \ref{lem_DPF}}.) 

    Assumption (3) follows from elliptic regularity and the upper bound on the elastic energy given by {$\int_{B(x,r)} \abs{e(u)}^2 \le \bar{C}_{\rm Ahlf}r$} within \eqref{eqn:AhlforsReg2}.
    {Precisely, in the interior of $\Omega\setminus K$}, for {any} ball $B(y,t) \subset B(x,r)$ which is disjoint from $K$, we have
    \begin{equation*}
        \sup_{B(y,t/2)} \abs{e(u)}^2 \leq C \dashint_{B(y,t)} \abs{e(u)}^2 \dd{x}
    \end{equation*}
    by interior Schauder estimate for solutions of elliptic systems, see \cite[Section 5]{Iu3}.  In \EEE combination with the aforementioned bound on \BBB the \EEE elastic energy, we have \eqref{ass:3}. {For balls centered {on} $K$ where $K$ is $C^{1,1/4}$ regular}, the details are more involved. {Letting  $B(y,t) \subset B(x,r)$  with  $y \in K$  and $K$ being $C^{1,1/4}$-regular in $B(y,t)$}, {a} similar Schauder estimate for solutions of elliptic systems with a  Neumann boundary condition {holds}. In {this case}, the estimate can be proven by straightening $K$ via a change of variable and extending the solution on the other side by reflection (the coefficients of the equations are reflected accordingly). Then the problem {reduces to} an interior regularity {estimate} which can be dealt with {via a standard} `freezing the coefficients' trick.
    This procedure is detailed in the scalar case in \cite[Theorem~7.53]{Ambrosio-Fusco-Pallara:2000} but also works for general elliptic systems. In particular, it has been adapted to Lamé's equations in \cite[Theorem 3.18]{FFLM07}. 
    {There, the} authors prove that weak solutions of Lamé's equations with a Neumann boundary condition are Hölder differentiable up to the boundary but the proof also yields the above Schauder estimate (see \cite[Lemma A.1]{CL2} for the missing details).
\end{proof}

\subsection{Initialization of the jump}\label{sec: look at jumo}

We start by introducing  an {auxiliary} notion {of jump, which has implicitly made an appearance in the proof of Theorem \ref{thm: david}}.  Besides the normalized jump defined in \eqref{def:normalizedJump},   for $B(x_0,r_0) \subset \Omega$ with $x_0 \in K$ {and $\beta_K(x_0,r_0) \leq 1/2$}, \EEE we define the   \emph{normalized  affine jump} by
\begin{equation*}
    {J}_{\rm aff}(x_0,r_0) =  \frac{\abs{a^+_{x_0,r_0}(x_0)  - a^-_{x_0,r_0}(x_0)} + r_0| A^+(x_0,r_0) -  A^-(x_0,r_0)|}{\sqrt{r_0}},
\end{equation*}
where $a^\pm_{x_0,r_0}(y):=A^\pm(x_0,r_0)\,y + b^\pm  (x_0,r_0)$ denote the corresponding rigid motions.  Note that the essential notion for us is $J(x_0,r_0)$, \cl{defined in \eqref{def:normalizedJump},} but in a first step it will be easier to control $J_{\rm aff}(x_0,r_0)$  compared to $J$, i.e.,  $J_{\rm aff}(x_0,r_0)$ \EEE  serves as an intermediate auxiliary notion.

We first start with  this first step giving control  on \EEE $J_{\rm aff}$ \BBB and \EEE address a bound on $J$ afterwards.   Recall also the normalized $p$-elastic energy introduced in (\ref{eqn:pelastic}).

\begin{lemma}[Initializing $J_{\rm aff}$]\label{lem:forcedJump}
    Let $p \in (1,2]$ and $(u,K)$ be a Griffith almost-minimizer \BBB  with respect to a gauge $h$. \EEE There are constants $\eps_0 \in (0,1/2)$ and $c_{\rm jump} > 0 $, depending only on $\mathbb{C}$, such that, for all $x_0 \in K$, $r_0 > 0$ such that $B(x_0,r_0)   \subset \Omega$ and
    \begin{equation*}
        \beta(x_0,r_0) + \omega_p(x_0,r_0) + h(r_0) \leq \eps_0,
    \end{equation*}
    it holds
    \begin{equation}\label{eqn:initJump}
        J_{\rm aff}(x_0,r_0 ) \geq c_{\rm jump} .
    \end{equation}
\end{lemma}

\begin{proof}
    To obtain the result, one supposes that \eqref{eqn:initJump} does not hold, and then constructs an  energetically favorable competitor by simply clearing out the crack in the interior and introducing a smooth transition, which is paid for in the elastic energy, and a small addition to the crack in the form of a wall-set.  
    As the proof of this result essentially follows \ks{Step 5} of Theorem \ref{thm: david}  by starting from two arcs $I_1$ and $I_2$ given by $I_1= \partial B(0,\rho) \cap D_{\mu r_0}^+ $ and $I_2= \partial B(0,\rho) \cap D_{\mu r_0}^- $ \BBB (recall \eqref{eq :Bx0}), \EEE in the spirit of brevity, we do not provide further details.     
\end{proof}

The idea in the proof of Lemma \ref{lem_init_J} is to use  Lemma \ref{lem_aA} to obtain a good control for  $J$ on  sufficiently small balls.
Property \eqref{eqn:initJump}, however, is not enough {since $J_{\rm aff}$ is possibly large due to the difference in rigid motions and not the difference in average values hidden in $J$. More precisely,} one could have
$|a^+_{x_0,t}(x_0)  - a^-_{x_0,t}(x_0)| = 0$ and  $| A^+(x_0,t) -  A^-(x_0,t)| \ge c_{\rm jump}/\sqrt{t}$ for $t \le r_0$, and in this case {a lower bound on the growth of $t \mapsto |a^+_{x_0,t}(x_0)  - a^-_{x_0,t}(x_0)|$ cannot be obtained since it is too small initially.} As we show now, this phenomenon can be avoided by slightly changing the center, and then $J$ can be estimated via Lemma \ref{lem_aA}.

\begin{proof}[Proof of Lemma \ref{lem_init_J}]
    {Let $x_0 \in K$, $r_0 > 0$ be such that $B(x_0,r_0) \subset \Omega$ and assume that
        \begin{equation}\label{eq: 6.3bis}
            \beta(x_0,r_0) + \omega_p(x_0,r_0) + h(r_0) \leq \varepsilon_0, 
        \end{equation}
    for some $\varepsilon_0 > 0$ which will be fixed during the proof, only depending on $C_0$ and $\mathbb{C}$.}
    We split the proof into three steps.  In the first two steps, we prove a weak estimate on $J$, namely:   there  exist two constants $c_0,\bar{c}$, depending only \BBB on \EEE $C_{\rm Ahlf}$,  as well as         $y_0 \in B(x_0,r_0/2) \cap K$ and  $r \in [\bar{c}r_0, r_0/4]$ such that  $B(y_0,r) \subset B(x_0,r_0)$ and 
    \begin{align}\label{eq: before smaller}
        J(y_0,r) \geq c_0c_{\rm jump},  
    \end{align}
    where $c_{\rm jump}$ is the constant of {Lemma \ref{lem:forcedJump}.}  In the third step, we conclude the proof. To show \eqref{eq: before smaller}, we fix  $ 0 < \kappa \le 1/2$ which will be chosen sufficiently small depending only on $C_{\rm Ahlf}$ {(see \eqref{eq: end of the step} below) and distinguish between the two cases where}   
    \begin{align}\label{eq: the smallness case0}
        \inf_{y \in K\cap B(x_0,r_0/4)}\abs{(A^+(x_0,r_0) - A^-(x_0,r_0))y +   (b^+(x_0,r_0) - b^-(x_0,r_0))} \ge  \kappa c_{\rm jump} \sqrt{r_0}
    \end{align} 
    holds or not. Without restriction we choose $\eps_0 \le \eps_{\rm Ahlf}$ {(see \eqref{eqn:AhlforsReg}) and} smaller than the constant in Lemma \ref{lem:forcedJump}. Along the proof, we further choose $\eps_0$ small enough depending  on $C_0$  (see \eqref{CCC0}), \EEE $c_{\rm jump}$, and $C_{\rm Ahlf}$ (and thus only on $C_0$ and $\mathbb{C}$).

    \emph{Step 1: Proof of \eqref{eq: before smaller} if \eqref{eq: the smallness case0} holds.} In this case, to control $J$ defined in \eqref{def:normalizedJump}, we estimate the difference of rigid motions on balls of different size. {Letting without restriction  $\eps_0\leq 1/32$,} by  \eqref{eq: 6.3bis} and Remark \ref{rmk_beta}  we have  $\beta(x_0,t) \le 1/8$ for all $t \in [r_0/4,r_0]$, and then   by Lemma \ref{lem_aA} we get 
    $$    | a^\pm_{x_0,r_0}(y) -  a^\pm_{x_0,r_0/4}(y)|    \leq C r_0^{1/2}\omega_p(x_0,r_0)^{1/4}   \leq C r_0^{1/2}\eps_0^{1/4}  \le  \frac{1}{4} \kappa c_{\rm jump} \sqrt{r_0} $$
    for all $y \in B(x_0,r_0/4)$, where the last step follows for $\eps_0$ small enough.   This along with \eqref{eq: the smallness case0} shows
    $$\inf_{y \in K\cap B(x_0,r_0/4)}\abs{a^+_{x_0,r_0/4}(y) - a^-_{x_0,r_0/4}(y) } \ge \frac{1}{2}  \kappa c_{\rm jump} \sqrt{r_0}, $$
    and therefore \eqref{eq: before smaller} holds  for $y_0=x_0$, $r =  r_0/4$, and $c_0 = \kappa/2$. 

    \emph{Step 2: Proof of \eqref{eq: before smaller} if \eqref{eq: the smallness case0} does not hold.} When \eqref{eq: the smallness case0} does not hold,  there exists some $\bar{y}\in {K \cap}B(x_0,r_0/4)$ such that 
    \begin{align}\label{eq: the smallness case1}
&\abs{(A^+(x_0,r_0) - A^-(x_0,r_0))\bar{y} +   (b^+(x_0,r_0) - b^-(x_0,r_0))} <  \kappa c_{\rm jump} \sqrt{r_0}.
    \end{align}
    We  now use the estimate $J_{\rm aff}(x_0, r_0  ) \geq c_{\rm jump}$ from  Lemma \ref{lem:forcedJump} to prove that 
    \begin{align}\label{eq: the smallness case2}
        \abs{A^+(x_0,r_0) - A^-(x_0,r_0)} \ge  \frac{1}{4}c_{\rm jump} / \sqrt{r_0}.
    \end{align}
    Indeed, as $|\bar{y}  -x_0| \le r_0/4$, we compute   by  Lemma \ref{lem:forcedJump} and \eqref{eq: the smallness case1}
    \begin{align*} 
        c_{\rm jump}\sqrt{r_0} &\le \abs{a^+_{x_0,r_0}(x_0)  - a^-_{x_0,r_0}(x_0)} + r_0| A^+(x_0,r_0) -  A^-(x_0,r_0)| \\ & \le \abs{a^+_{x_0,r_0}(\bar{y})  - a^-_{x_0,r_0}(\bar{y})} + (r_0 + r_0/4)| A^+(x_0,r_0) -  A^-(x_0,r_0)| \\ &\le \kappa c_{\rm jump} \sqrt{r_0} +  2r_0| A^+(x_0,r_0) -  A^-(x_0,r_0)|.  
    \end{align*}
    As $\kappa \le 1/2$, this shows \eqref{eq: the smallness case2}.

    By Ahlfors-regularity \eqref{eqn:AhlforsReg}, we can find $y_0 \in K \cap B(\overline{y},r_0/4)$ with $|y_0-\bar{y}| \ge 2\bar{c}r_0$ for some   small constant \EEE $0 <\bar{c} < 1/16$ only depending on $C_{\rm Ahlf}$. {Explicitly, if this were not the case, then $K \cap B(\bar{y},r_0/4) = K \cap B(\bar{y},2\bar c r_0)$. With this, we have 
    $$C_{\rm Ahlf}^{-1}r_0/4 \leq \HH^1\big(K\cap B(\bar{y},r_0/4)\big) = \HH^1\big(K\cap B(\bar{y},2\bar{c}r_0)\big)\leq C_{\rm Ahlf} 2\bar{c}r_0 ,$$ which is a contradiction for sufficiently small $\bar{c}$.}
    Note that in particular $y_0 \in K \cap B(x_0,r_0/2)$.
    Our goal is now to show that  for all $y \in K \cap B(y_0,\bar{c}r_0)$
    \begin{align}\label{eq: to show}
        \abs{(A^+(y_0,r) - A^-(y_0,r))y + (b^+(y_0,r) - b^-(y_0,r))} \ge  c_0 c_{\rm jump} \sqrt{r} \quad \quad 
    \end{align}
    for $c_0$ {depending only} on $C_{\rm Ahlf}$,  where {$r :=  \bar{c}r_0$} for shorthand. This then shows \eqref{eq: before smaller}.

    To conclude Step 2, we show \eqref{eq: to show}. We  can estimate the difference of the rigid motions by varying center and radius by Remark~\ref{rem: vary}, where {$\beta(x_0,r_0) \leq r_0/(16 r)$} is guaranteed by  \eqref{eq: 6.3bis} with $\eps_0$ taken smaller depending only on $C_{\rm Ahlf}$. With $\omega_p(x_0,r_0)\le \eps_0$, we {have}
    \begin{align}\label{eq: two pieces1}
        r_0\big|A^\pm(y_0,r) - A^\pm(x_0,r_0)\big|  \le Cr_0^{1/2}\eps_0^{1/4}, 
    \end{align}
    \begin{align}\label{eq: two pieces2}
        \big|(A^\pm(y_0,r) - A^\pm(x_0,r_0))z    +  (b^\pm(y_0,r) - b^\pm(x_0,r_0))\big|  \le Cr_0^{1/2}\eps_0^{1/4} \quad \text{for all $z \in B(x_0,r_0) $}
    \end{align}
    for $C$ depending on $r_0/r$ and thus on  $C_{\rm Ahlf}$. 
    (Note that by the remark we get \eqref{eq: two pieces2} only on $B(y_0,r)$, but the estimate for points in $B(x_0,r_0)$ readily follows for a bigger constant by \eqref{eq: two pieces1}.)  We now use  the triangle inequality several times to conclude. Fix $y \in B(y_0,r)$.  First, by \eqref{eq: two pieces2} {with} $z =\bar{y}$ we get 
    \begin{align*}
        d_{y,r}:= &\abs{ (A^+(y_0,r) - A^-(y_0,r))y + (b^+(y_0,r) - b^-(y_0,r))}\\
                  &\ge  \big| (A^+(y_0,r) - A^-(y_0,r))y + (b^+(x_0,r_0) - b^-(x_0,r_0))  \\ 
                  & \ \ \ \ \ \  -   (A^+(y_0,r) - A^+(x_0,r_0))\bar{y}  + (A^-(y_0,r) - A^-(x_0,r_0))\bar{y} \big|  - Cr_0^{1/2}\eps_0^{1/4} .
    \end{align*}
    Then, employing   \eqref{eq: the smallness case1} we find 
    \begin{align*}
        d_{y,r} & \ge  \big| (A^+(y_0,r) - A^-(y_0,r))y     -   A^+(y_0,r)\bar{y}  + A^-(y_0,r)\bar{y}\big|  -  Cr_0^{1/2}\eps_0^{1/4} - \kappa c_{\rm jump} \sqrt{r_0} \\ & = \big| (A^+(y_0,r) - A^-(y_0,r))(y-\bar{y})   \big|  -  Cr_0^{1/2}\eps_0^{1/4} - \kappa c_{\rm jump} \sqrt{r_0}.
    \end{align*}
    {As {$A^+(y_0,r) - A^-(y_0,r)$} is a $2 \times 2$ skew-symmetric matrix and $$|y - \bar{y}| \ge |y_0-\bar{y}|  - |y - y_0|  \ge  2\bar{c} r_0 - \bar{c} r_0 =  \bar{c} r_0,$$ we see that
        \begin{align*}
            \abs{\left(A^+(y_0,r) - A^-(y_0,r)\right)(y - \bar{y})} &= \frac{1}{\sqrt{2}} \abs{A^+(y_0,r) - A^-(y_0,r)} \abs{y - \bar{y}} \geq \frac{\bar{c} r_0}{\sqrt{2}} \abs{A^+(y_0,r) - A^-(y_0,r)}.
    \end{align*}}
    Then by \eqref{eq: the smallness case2} and \eqref{eq: two pieces1}  we find
    \begin{align}\label{eq: end of the step}
        d_{y,r} & \ge  \frac{\bar{c} \sqrt{r_0}}{4\sqrt{2}}c_{\rm jump}    -  \frac{2  C\bar{c}r_0}{\sqrt{2}} \eps_0^{1/4} r_0^{-1/2}  - Cr_0^{1/2}\eps_0^{1/4}- \kappa c_{\rm jump} \sqrt{r_0}  \notag \\
                & \ge \frac{ \sqrt{\bar{c}} \sqrt{r}}{4\sqrt{2}} c_{\rm jump}    -   C r_0^{1/2}\eps_0^{1/4}   - \frac{\kappa}{\sqrt{\bar{c}}} c_{\rm jump} \sqrt{r},
    \end{align}
    where in the second line we used a bigger constant $C$ and $r = \bar{c}r_0$. For $\eps_0$ and $\kappa$ small enough (depending only on $\bar{c}$, $c_{\rm jump}$, $C$ and thus only on $\mathbb{C}$), we obtain  \eqref{eq: to show} for $c_0 = \sqrt{\bar{c}}/8$ and consequently \eqref{eq: before smaller}, as desired.  This concludes Step 2 of the proof.

    \emph{Step 3: Proof of the statement.} Let us now show the statement. To this end, let   $y_0 \in B(x_0,r_0/2) \cap K$ and  $r \in [\bar{c}r_0, r_0/4]$ such that  \eqref{eq: before smaller} holds.  We choose $\gamma \in (0,\bar{c})$  small enough (depending on $c_0, \bar{c}$ given in \eqref{eq: before smaller},  $ c_{\rm jump}$ and $C_0$) such that
    \begin{align}\label{eq: sigma}
        \frac{1}{2}  \bar{c}^{1/2}c_0c_{\rm jump}    \gamma^{-1/2} \ge  C_0.
    \end{align}
    We now choose $\eps_0$ sufficiently small (depending also on $\gamma$) such that
    \begin{align}\label{eq: eps00}
        C     \eps_0^{1/4}   /\bar{c}^{ 1/2 \EEE} \EEE  \le \frac{1}{4} c_0c_{\rm jump}   , \quad \quad \eps_0 \le \frac{\gamma}{16}, 
    \end{align}
    where $C$ denotes the constant in \eqref{eqn:bAControl}. 
    {As $\beta(x_0,r_0) \leq \varepsilon_0$ and $y_0 \in B(x_0,r_0/2)$, we know by Remark~\ref{rmk_beta} that for all $t \in [\gamma r_0,r_0/2]$ we can control $\beta(y_0,t) \leq 2 \gamma^{-1} \varepsilon_0$, and thus  $\beta(y_0,t) \leq 1/8$  by \eqref{eq: eps00}.}
    {Therefore, we can} apply (\ref{eqn:bAControl}) for $t \in [\gamma r_0,r]$ to find by \eqref{eq: 6.3bis},  \eqref{eq: eps00},  \eqref{eqn:naiveUpEst}, $p > 4/3$, and $r \geq \bar{c} r_0$ that
    \begin{equation*}
        \abs{a^\pm_{y_0,r}(y) - a^\pm_{y_0,t}(y) } \leq C r^{1/2} {\omega_p}(x_0,r)^{1/4} \le C {r^{1/2}}      \eps_0^{1/4}   / \bar{c}^{ 1/2 \EEE} \EEE  \le \frac{1}{4} c_0c_{\rm jump}  \sqrt{r} \quad \text{for all $y \in B(y_0,t)  $}.
    \end{equation*}
    The estimate  $ J(y_0,r) \geq c_0c_{\rm jump}$    established in \eqref{eq: before smaller}  can be equivalently written as  $\abs{a^+_{y_0,r}(y)  - a^-_{y_0,r}(y)} \ge c_0c_{\rm jump}  \sqrt{r}$ for all ${y} \in K \cap B(y_0,r)$, see \eqref{def:normalizedJump}.   This   implies 
    $\abs{a^+_{x_0,t}(y)  - a^-_{x_0,t}(y)} \ge \frac{1}{2}c_0c_{\rm jump}  \sqrt{r}$ for all $y \in K \cap B(y_0,t)$, and therefore by \eqref{eq: sigma} and $r \ge \bar{c}r_0$  we derive for $t= \gamma r_0$ that
    $${J}(y_0,t) \ge  \frac{1}{2}c_0c_{\rm jump}  r^{1/2} t^{-1/2} \ge  \frac{1}{2}   \bar{c}^{1/2}c_0c_{\rm jump}    r_0^{1/2} t^{-1/2} =     \frac{1}{2}  \bar{c}^{1/2}c_0c_{\rm jump}   \gamma^{-1/2} \ge  C_0.  $$    
    This concludes the proof. 
\end{proof}

\section*{Acknowledgements} 
This work was supported by the DFG project FR 4083/3-1 and by the Deutsche Forschungsgemeinschaft (DFG, German Research Foundation) under Germany's Excellence Strategies EXC 2044 -390685587, Mathematics M\"unster: Dynamics--Geometry--Structure and EXC 2047/1 - 390685813, and NSF (USA) RTG grant DMS-2136198. The majority of this work was completed while K.S. was affiliated with the Hausdorff Center for Mathematics at the University of Bonn. K.S. also thanks Sergio Conti for insightful discussions on related problems.


\typeout{References}

%
%

\end{document}